\documentclass[final,hidelinks,onefignum,onetabnum]{siamart251216}

\usepackage[utf8]{inputenc}
\usepackage{amsmath,amssymb,mathtools}
\usepackage{mathrsfs}
\usepackage{bm}
\usepackage{enumitem}
\usepackage{color}
\usepackage{cite}
\usepackage{tikz}
\usetikzlibrary{arrows.meta,positioning,fit,calc}

\newcommand{\E}{\mathbb{E}}
\newcommand{\Var}{\mathrm{Var}}
\newcommand{\Cov}{\mathrm{Cov}}
\newcommand{\R}{\mathbb{R}}

\newcommand{\norm}[1]{\left\lVert #1 \right\rVert}
\newcommand{\abs}[1]{\left\lvert #1 \right\rvert}

\newsiamthm{assumption}{Assumption}
\newsiamremark{remark}{Remark}
\crefname{assumption}{Assumption}{Assumptions}
\Crefname{assumption}{Assumption}{Assumptions}
\crefname{remark}{Remark}{Remarks}
\Crefname{remark}{Remark}{Remarks}

\headers{Neural control variates for kinetic equations}{Wei Chen, Giacomo Dimarco, Lorenzo Pareschi}

\title{Control variates with neural surrogates for uncertainty quantification in kinetic equations}
\author{Wei Chen\thanks{School of Mathematical Sciences, Xiamen University, China (weichenmath@stu.xmu.edu.cn).}
\and
Giacomo Dimarco\thanks{Department of Mathematics and Computer Science, University of Ferrara, Italy (giacomo.dimarco@unife.it).}
\and
Lorenzo Pareschi\thanks{Maxwell Institute for Mathematical Sciences and Department of Mathematics, School of Mathematical and Computer Sciences, Heriot-Watt University, Edinburgh, UK (l.pareschi@hw.ac.uk). Also
	affiliated with Department of Mathematics and Computer Science, University of Ferrara, Italy.}}

\ifpdf
\hypersetup{
  pdftitle={Control variates with neural surrogates for uncertainty quantification in kinetic equations},
  pdfauthor={Wei Chen, Giacomo Dimarco, Lorenzo Pareschi}
}
\fi

\begin{document}

\maketitle

\begin{abstract}
Efficient uncertainty quantification for kinetic equations with random inputs is challenging because it requires repeated simulations of high-dimensional models, such as the Boltzmann, Landau, and related collisional equations, whose computational cost can quickly become prohibitive. Multifidelity control variates address this difficulty by coupling a small number of high-fidelity simulations with many evaluations of lower-complexity reduced models. In this work, we analyze the case in which the reduced model is replaced by a neural surrogate rather than evaluated through a classical numerical scheme. We show that the
resulting estimator remains unbiased and that the change in the optimal
variance induced by the neural approximation is controlled by the error between the exact low-fidelity observable and its
neural approximation.
This estimate is then combined with residual stability estimates for inhomogeneous Fokker--Planck and Bhatnagar--Gross--Krook surrogates. We also extend the analysis to several control variates and to an asymptotic-preserving (AP) hierarchy containing the limiting Euler observable. In the fluid limit, the optimal hierarchical variance converges to the variance associated with the limiting Euler control, while the contribution of the intermediate kinetic correction vanishes. Numerical tests based on micro--macro neural surrogates illustrate the predicted variance stability and the behavior of the AP hierarchy.
\end{abstract}



\section{Introduction}

The development of efficient numerical methods for kinetic equations with
random inputs is challenging because of the high dimensionality of the
problem and the multiscale nature of the underlying models
\cite{Villani2002survey,DimarcoPareschi2019,DimarcoPareschi2020,
HuJin2016,MedagliaPareschiZanella2023,
DimarcoPareschiZanella2024NFP,dimarco2014numerical}.
A broad class of collisional kinetic equations can be written in the form
\begin{equation}
\partial_t f+v\cdot\nabla_x f+F[f]\cdot\nabla_v f
=
\frac{1}{\varepsilon}\mathcal C(f,f),
\label{eq:kinetic}
\end{equation}
where $f=f(t,x,v;z)$, $z$ denotes the random input and $\varepsilon>0$ is a
Knudsen-type scaling parameter. The force term may be absent, as in the
classical Boltzmann equation, or coupled to the distribution function
through a field equation, as in plasma applications. For Boltzmann and
Landau dynamics, the collision operator $\mathcal C(f,f)$ is nonlinear and
has an integral or integro-differential structure in the velocity
variable \cite{Villani2002survey}. Its evaluation represents a major computational bottleneck
of the numerical solution, a difficulty that becomes particularly
severe in uncertainty quantification, where the kinetic equation has to
be solved for a large number of realizations of the random input \cite{dimarco2014numerical}.

Monte Carlo methods are attractive in this setting, since they can be
combined with existing deterministic solvers and are naturally parallel.
Their slow convergence, however, makes variance reduction essential when
the high-fidelity model is a Boltzmann, Landau, or related collisional
kinetic equation. Multiscale control variate methods reduce this cost by
combining a small number of high-fidelity simulations with many
evaluations of cheaper models that remain strongly correlated with the
quantity of interest. In kinetic equations, these lower-fidelity
descriptions arise naturally from simplified collision mechanisms,
asymptotic approximations, or macroscopic fluid closures
\cite{DimarcoPareschi2019,DimarcoPareschi2020}. This approach fits within
the general framework of multifidelity Monte Carlo methods
\cite{PeherstorferWillcoxGunzburger,
PeherstorferWillcoxGunzburger2018,GorodetskyGeraciEldredJakeman};
related multilevel techniques for the BGK equation were considered in
\cite{HuPareschiWang2021}. We also mention intrusive methods for kinetic equations based on
generalized polynomial chaos and stochastic Galerkin discretizations
\cite{HuJin2016,MedagliaPareschiZanella2023,
DimarcoPareschiZanella2024NFP}.

Neural surrogates provide a further possibility within this framework.
Instead of solving the reduced model for every realization of the random
input, one may train an approximation of the reduced dynamics and use
the corresponding observable as a control variate. This requires a low-fidelity model that is not only inexpensive
and sufficiently correlated with the high-fidelity dynamics, but also
suitable for an efficient residual-based training
\cite{RaissiPerdikarisKarniadakis2019}. A direct neural approximation of the
Boltzmann or Landau equation would require repeated evaluations of the
full nonlinear integral or integro-differential collision operator at
the collocation points and throughout the optimization. We therefore
consider reduced kinetic models based on BGK relaxation or nonlinear
Fokker--Planck operators, whose local or differential structure permits
an efficient evaluation of the residual entering the loss while
retaining the relevant equilibrium and asymptotic properties.

The learned model is then coupled with the conventional high-fidelity
solver through the control-variate estimator, rather than used as its
replacement. Consequently, the neural approximation does not introduce
a bias in the estimate of the high-fidelity expectation: its error
affects instead the efficiency of the variance reduction, through the
correlation between the neural and high-fidelity observables. This
provides a natural way to combine neural surrogates with standard
numerical methods while retaining a quantitative control of their
effect on the estimator.

Neural control variates have also been considered in other Monte Carlo
settings
\cite{WanZhongXiongZhu2020,MullerRousselleKellerNovak2020}.
Neural approximations for kinetic and multiscale equations, including
asymptotic-preserving formulations, have recently been proposed for
several kinetic models
\cite{BertagliaLuPareschiZhu2022,JinMaWu,
LiuWangZhuZhu2025,LiWangLiuWangDong2024,WuYanJinMa2024,LouMengKarniadakis2021,BoisFranckNavoretVigon2022}.
The numerical strategies developed in
\cite{PareschiChenDimarco,ChenDimarcoPareschiSAPNN} showed that
structure-preserving and micro--macro tensor neural surrogates can be
used as effective controls in rarefied-gas and plasma applications. 

The purpose of the present work is to analyze how the error in the
neural approximation of the reduced model propagates to the
multifidelity estimator.
We quantify how the error between the exact reduced
observable and its neural approximation affects the optimal coefficient
and the variance of the multifidelity estimator. Weakly correlated
surrogates are automatically downweighted by the covariance-based
coefficient, whereas sufficiently correlated surrogates can produce a
substantial reduction of variance. The analysis is perturbative with
respect to the low-fidelity control and does not exclude the case, observed
in computations, in which a calibrated or enriched neural surrogate is
more strongly correlated with the high-fidelity observable than the
original low-fidelity model.

A further issue arises in multiscale regimes. When the Knudsen number is
small, the kinetic description approaches a fluid regime and suitable
macroscopic quantities evolve according to a limiting system, such as
the Euler equations. In this case, the variance-reduction mechanism
should remain compatible with the fluid limit. This leads to an
asymptotic-preserving interpretation at the level of the estimator: if
the high-fidelity and low-fidelity observables converge to the same Euler
observable, then the hierarchical control should concentrate on the
limiting level as the Knudsen number tends to zero. This complements the
usual deterministic notion of asymptotic preservation
\cite{DP_EXP,DP_IMEX,DP_Multi}.

We show that replacing an exact low-fidelity observable by its neural
approximation preserves unbiasedness and that the change in the optimal
variance is controlled by their $L^2(\Omega)$ distance. The
result is extended to several control variates. For kinetic quantities
of interest, the observable error is estimated through residual
stability bounds for inhomogeneous nonlinear Fokker--Planck and BGK
reduced models in a perturbative weighted Sobolev setting. The
high-fidelity Boltzmann or Landau equation enters only through the random
observable and its covariance with the controls, while the PDE analysis
concerns the reduced model and its neural approximation. Finally, when the Euler observable is included explicitly in the
hierarchy, we characterize the fluid-limit variance without requiring
uniform nondegeneracy of the covariance matrix. The Euler control
becomes asymptotically dominant, while the contribution of the
intermediate kinetic correction vanishes.

The rest of the paper is organized as follows. Section~2 contains the
probabilistic stability estimate for approximate control variates and
its extension to multiple controls. Section~3 applies the estimate to
Fokker--Planck and BGK reduced models and studies the AP regime.
Section~4 reports numerical evidence of the variance-stability mechanism
and the AP hierarchy. Section~5 concludes the paper.

\section{Multifidelity estimators with neural control variates}\label{sec:setting}

We begin with the probabilistic analysis, independently of the specific kinetic equation. The uncertain input is described by a probability space $(\Omega,\mathcal F,\mathbb P)$, and $z:\Omega\to Z\subseteq\mathbb R^{d_z}$, $d_z\geq1$, denotes the random parameter. Quantities of interest are regarded as random variables in $L^2(\Omega)$. The scalar observables below may represent individual spatial degrees of freedom, cell averages, or components of macroscopic moment fields; norms of the reconstructed expected fields are taken after the statistical estimation step. We quantify how replacing an exact low-fidelity control by an approximate, possibly neural, observable affects the optimal coefficient and variance. The kinetic estimates used to bound the resulting observable error are developed in Section~\ref{sec:applications}.

\subsection{Two-level multifidelity estimator}\label{sec21}
We first recall the standard Monte Carlo setting. Let $Q_H\in L^2(\Omega)$
be a scalar high-fidelity quantity of interest depending on the uncertain
input. The objective is to approximate its expectation $\mathbb E[Q_H]$.
Given $M$ independent realizations $Q_H^{(1)},\ldots,Q_H^{(M)}$, the
standard Monte Carlo estimator is
\[
E_{M}[Q_H]
:=
\frac{1}{M}\sum_{k=1}^{M} Q_H^{(k)} .
\]
It is unbiased and satisfies
\[
\mathbb E[E_{M}[Q_H]]
=
\mathbb E[Q_H],
\qquad
\operatorname{Var}(E_{M}[Q_H])
=
\frac{\operatorname{Var}(Q_H)}{M}.
\]
Thus the statistical error decreases only as $M^{-1/2}$. When $Q_H$
is generated by an expensive kinetic model, such as a Boltzmann or Landau
model, reducing the sampling error by simply increasing $M$ is typically
computationally prohibitive. A control-variate method reduces this cost by exploiting an additional random
variable that is correlated with $Q_H$, but cheaper to evaluate. In a
two-level multifidelity setting, $Q_H$ denotes the high-fidelity observable,
whereas $Q_L$ denotes an ideal low-fidelity control variate. The term
multifidelity refers to the simultaneous use of models or observables with
different levels of accuracy and computational cost: a small number of
high-fidelity samples is combined with many evaluations of a cheaper
low-fidelity model. The low-fidelity observable is not used as a 
replacement for $Q_H$. It is used statistically, through its covariance with
$Q_H$, to reduce the variance of the estimator for $\mathbb E[Q_H]$. This requires coupled evaluations of the high- and low-fidelity observables for
the same realizations of the random input, so that their common fluctuations can
be estimated and subtracted.

The main idea in this work is that the ideal low-fidelity control $Q_L$ can be approximated by a neural surrogate
$Q_N$ obtained from the reduced model. We then ask how this
replacement affects unbiasedness, the optimal control-variate coefficient, and
variance reduction in the multifidelity estimator for $\mathbb E[Q_H]$. In the sequel, we consider high-fidelity, low-fidelity, and neural
quantities of interest generated from the kinetic densities
\cite{PareschiChenDimarco,ChenDimarcoPareschiSAPNN}
\[
f_H(t,x,v;z),\qquad
f_L(t,x,v;z),\qquad
f_N(t,x,v;z),
\]
depending on time $t>0$, phase-space variables
$(x,v)\in\mathbb R^{d_x}\times\mathbb R^{d_v}$, where $d_x$ is the space
dimension and $d_v$ is the velocity dimension, and the uncertain input $z$.
The corresponding scalar observables are
\[
Q_H(z)=\mathcal Q(f_H(\cdot,\cdot,\cdot;z)),\qquad
Q_L(z)=\mathcal Q(f_L(\cdot,\cdot,\cdot;z)),\qquad
Q_N(z)=\mathcal Q(f_N(\cdot,\cdot,\cdot;z)).
\]
The notation above suppresses the variables on which the observable may depend.
For instance, $\mathcal Q$ may be a pointwise evaluation of the distribution
at a fixed $(t,x,v)$, a spatial cell average, or a moment functional evaluated
at a fixed time and position. Here $f_H$ denotes the expensive kinetic model. Its
precise equation is not needed in the analysis below. The high-fidelity model
enters only through the random variable $Q_H$ and through its covariance with
the control variates. Throughout the paper we assume
\[
Q_H,Q_L,Q_N\in L^2(\Omega).
\]
For clarity, we first present the error analysis for a single control variate.
The extension to several simultaneous control variates is treated separately in
Section~\ref{sec:multiple-cv}, where variances and covariances are replaced by
covariance matrices and vectors \cite{DimarcoPareschi2020}.
With this
notation, the control-variate estimator of $\mathbb E[Q_H]$ based on $Q_i$,
$i\in\{L,N\}$, is
\begin{equation}\label{eq:ideal-cv-estimator}
\begin{split}
	\mathcal E_i^\lambda[Q_H]
	&:=
	E_{M_H}[Q_H-\lambda Q_i]
	+
	\lambda E_{M_L}[\widetilde Q_i]\\
	&=
	\frac{1}{M_H}\sum_{k=1}^{M_H}
	\left(Q_H^{(k)}-\lambda Q_i^{(k)}\right)
	+
	\lambda
	\frac{1}{M_L}\sum_{j=1}^{M_L}
	\widetilde Q_i^{(j)} .
	\end{split}
\end{equation}
Here $\{(Q_H^{(k)},Q_i^{(k)})\}_{k=1}^{M_H}$ are i.i.d. samples of the
coupled pair, while $\{\widetilde Q_i^{(j)}\}_{j=1}^{M_L}$ are i.i.d. samples
of $Q_i$, $i\in\{L,N\}$, independent of the first family. Throughout the analysis, the neural surrogate and the coefficient are regarded as fixed independently of these two sample families. If they are constructed from independent training or pilot samples, the formulas below hold conditionally on such data. The parameter $\lambda\in\mathbb R$ is the control-variate coefficient. Its optimal value
is obtained below by minimizing the estimator variance. In a pointwise
formulation it may depend on $(t,x,v)$, whereas for macroscopic moments it
depends on the corresponding moment component and on the remaining physical
variables. In applications, the low-fidelity or neural surrogate is sampled
much more extensively than the high-fidelity model, so that typically
$M_L\gg M_H$. The following standard control-variate identities are recalled for completeness
(see, for instance,
\cite{PeherstorferWillcoxGunzburger2018,GorodetskyGeraciEldredJakeman} for the
general multifidelity setting and
\cite{DimarcoPareschi2019,DimarcoPareschi2020} for kinetic applications).

\begin{proposition}[Two-level estimator and optimal variance]
\label{prop:exact-two-layer-variance}
For every $\lambda\in\mathbb R$, the control-variate estimator in \eqref{eq:ideal-cv-estimator} is unbiased:
\[
\E[\mathcal E_i^\lambda[Q_H]]=\E[Q_H],
\qquad i\in\{L,N\}.
\]
If $\Var(Q_i)>0$, the variance-minimizing control is
\begin{equation}\label{eq:lambda-star}
\lambda_i^*
=
\frac{\Cov(Q_H,Q_i)}{\Var(Q_i)}
\frac{M_L}{M_H+M_L},
\qquad i\in\{L,N\}.
\end{equation}
Consequently, defining the correlation coefficient
\[
\rho_{H,i}:=
\frac{\Cov(Q_H,Q_i)}
{\sqrt{\Var(Q_H)\Var(Q_i)}},
\]
the optimal estimator $\mathcal E_i^{*}[Q_H]:=\mathcal E_i^{\lambda^*_i}[Q_H]$
satisfies
\begin{equation}\label{eq:exact-two-layer-variance}
\Var(\mathcal E_i^{*}[Q_H])
=
\frac{\Var(Q_H)}{M_H}
\left(
1-\frac{M_L}{M_H+M_L}\rho_{H,i}^2
\right),
\qquad i\in\{L,N\}.
\end{equation}

\end{proposition}

\begin{proof}
Taking expectations in \eqref{eq:ideal-cv-estimator} gives
\[
\E[\mathcal E_i^\lambda[Q_H]]
=
\E[Q_H-\lambda Q_i]+\lambda\E[Q_i]
=
\E[Q_H].
\]
Since the two sample families in \eqref{eq:ideal-cv-estimator} are independent,
\begin{equation}\label{eq:variance}
\Var(\mathcal E_i^\lambda[Q_H])
=
\frac{1}{M_H}\Var(Q_H-\lambda Q_i)
+
\frac{\lambda^2}{M_L}\Var(Q_i),
\qquad i\in\{L,N\}.
\end{equation}
Expanding the right-hand side gives
\[
\Var(\mathcal E_i^\lambda[Q_H])
=
\frac{\Var(Q_H)}{M_H}
-
\frac{2\lambda}{M_H}\Cov(Q_H,Q_i)
+
\lambda^2
\left(
\frac{1}{M_H}+\frac{1}{M_L}
\right)
\Var(Q_i).
\]
Minimization with respect to $\lambda$ gives \eqref{eq:lambda-star}.
Substituting this value into the variance formula yields
\begin{equation}
\Var(\mathcal E_i^{*}[Q_H])
=
\frac{\Var(Q_H)}{M_H}
-
\frac{M_L}{M_H(M_H+M_L)}
\frac{\Cov(Q_H,Q_i)^2}{\Var(Q_i)}.
\label{eq:optvar}
\end{equation}
The representation \eqref{eq:exact-two-layer-variance} follows from the
definition of $\rho_{H,i}$.
\end{proof}
The quantity $\rho_{H,i}$ measures the strength of the linear statistical
dependence between the high-fidelity observable and the control. Values of
$|\rho_{H,i}|$ close to one correspond to an effective control variate,
whereas values close to zero give little improvement over standard Monte Carlo. When the two observables are strongly correlated, their random fluctuations are
similar and the corrected random variable has a much
smaller variance than $Q_H$ itself. Consequently, for a fixed number of
high-fidelity samples, the statistical error in the estimation of
$\mathbb E[Q_H]$ is reduced; equivalently, a prescribed accuracy can be
achieved with fewer expensive high-fidelity evaluations.

\begin{remark}
\label{rem:variance-loss-correlation}
Formula~\eqref{eq:exact-two-layer-variance} gives the exact difference between
the ideal and approximate optimal variances
\[
\operatorname{Var}(\mathcal E_N^{*}[Q_H])
-
\operatorname{Var}(\mathcal E_L^{*}[Q_H])
=
-\frac{\operatorname{Var}(Q_H)}{M_H}
\frac{M_L}{M_H+M_L}
\left(
\rho_{H,N}^2-\rho_{H,L}^2
\right).
\]
A calibrated or enriched neural surrogate, as shown in~\cite{PareschiChenDimarco, ChenDimarcoPareschiSAPNN}, may therefore outperform the original low-fidelity control variate when $\rho_{H,N}^2>\rho_{H,L}^2$.
\end{remark}

\begin{remark}
The quadratic variance representation~\eqref{eq:variance} also quantifies the effect of using a nonoptimal,
approximate, or estimated control-variate coefficient.
In fact, for every approximation $\tilde \lambda\approx \lambda_i^* $, we have
\begin{equation}
\Var(\mathcal E_i^{\tilde \lambda}[Q_H])
=
\Var(\mathcal E_i^*[Q_H])
+
\frac{M_H+M_L}{M_H M_L} \Var(Q_i) (\tilde \lambda-\lambda_i^*)^2 .
\end{equation}
Indeed, by \eqref{eq:variance},
\[
\Var(\mathcal E_i^{\tilde\lambda}[Q_H])
=
\frac{\Var(Q_H)}{M_H}
-
\frac{2\tilde\lambda \Cov(Q_H,Q_i)}{M_H}
+
\frac{M_H+M_L}{M_H M_L} \Var(Q_i) \tilde\lambda^2,
\]
completing the square and using the definition of $\lambda_i^*$ gives the identity. If $\tilde\lambda$ is estimated from an independent pilot sample, the same identity holds conditionally on the pilot data and isolates the additional variance due to coefficient estimation; see, for example,~\cite{PhamGorodetsky2022}.
\end{remark}

\subsection{Errors induced by the neural approximation}\label{sec:abstract}
We next quantify how the neural approximation affects the correlation
structure with $Q_H$. The relevant error is not the distance between the
neural surrogate and the high-fidelity quantity, but the error made in
replacing the exact low-fidelity control observable by its approximation.
The specific form depends on the neural architecture, the residual and training
errors of the surrogate, the moment evaluation, and possible AP errors. These
aspects are discussed in Section~\ref{sec:observables} and in
Remark~\ref{rem:architecture-residual-errors}; see also
\cite{MishraMolinaro,EirasBibiBunelDvijothamTorrKumar2024,JinMaWu,
LuWangXu2022,ChoNamYangYunHongPark2023,OhChoYunParkHong2025}.

The following elementary estimates are needed
below.

\begin{lemma}[Covariance and variance perturbations]
\label{lem:cov-perturb}
\label{lem:var-perturb}
We have
\begin{equation}\label{eq:cov-perturb}
\abs{\Cov(Q_H,Q_N)-\Cov(Q_H,Q_L)}
\le
\norm{Q_H-\E[Q_H]}_{L^2}\,\norm{Q_N-Q_L}_{L^2}.
\end{equation}
Moreover,
\begin{equation}\label{eq:var-perturb}
\abs{\Var(Q_N)-\Var(Q_L)}
\le
C\norm{Q_N-Q_L}_{L^2}
\bigl(
\norm{Q_L}_{L^2}+\norm{Q_N-Q_L}_{L^2}
\bigr).
\end{equation}
\end{lemma}

\begin{proof}
We have
\[
        \Cov(Q_H,Q_N)-\Cov(Q_H,Q_L)
        =
        \Cov(Q_H,Q_N-Q_L).
\]
Hence, by Cauchy--Schwarz, we get  \eqref{eq:cov-perturb}
\begin{align*}
\abs{\Cov(Q_H,Q_N-Q_L)}
&\le
\norm{Q_H-\E[Q_H]}_{L^2}
\norm{Q_N-Q_L-\E[Q_N-Q_L]}_{L^2}\\
&\le
\norm{Q_H-\E[Q_H]}_{L^2}
\norm{Q_N-Q_L}_{L^2}.
\end{align*}
Similarly,
\[
        \Var(Q_N)
        =
        \Var(Q_L)+2\Cov(Q_L,Q_N-Q_L)+\Var(Q_N-Q_L),
\]
and therefore
\[
\abs{\Var(Q_N)-\Var(Q_L)}
\le
2\abs{\Cov(Q_L,Q_N-Q_L)}+\Var(Q_N-Q_L).
\]
Applying Cauchy--Schwarz gives
\begin{align*}
\abs{\Var(Q_N)-\Var(Q_L)}
\le&\,
2\norm{Q_L-\E[Q_L]}_{L^2}\,\norm{Q_N-Q_L-\E[Q_N-Q_L]}_{L^2}\\
&+
\norm{Q_N-Q_L-\E[Q_N-Q_L]}_{L^2}^2.
\end{align*}
The final bound \eqref{eq:var-perturb} follows
from
\[
\norm{Q_N-Q_L-\E[Q_N-Q_L]}_{L^2}\le\norm{Q_N-Q_L}_{L^2},
\qquad
\norm{Q_L-\E[Q_L]}_{L^2}\le\norm{Q_L}_{L^2}.
\]
\end{proof}
We introduce now the following assumption:
\begin{assumption}[Nondegenerate control variate]\label{ass:nondegenerate-variance}
There exists $\sigma_0>0$ such that
\[
        \Var(Q_L)\ge \sigma_0,
        \qquad
        \Var(Q_N)\ge \frac{\sigma_0}{2}.
\]
\end{assumption}

The second lower bound is a perturbative condition. For instance, by
Lemma~\ref{lem:var-perturb}, it follows from
\[
C\norm{Q_N-Q_L}_{L^2}
\bigl(
\norm{Q_L}_{L^2}+\norm{Q_N-Q_L}_{L^2}
\bigr)
\le
\frac{\sigma_0}{2},
\]
because
\[
        \Var(Q_N)
        \ge
        \Var(Q_L)-\abs{\Var(Q_N)-\Var(Q_L)}.
\]
The following theorem shows that the exact variance comparison in
Remark~\ref{rem:variance-loss-correlation} is stable under small perturbations
of the low-fidelity observable.
\begin{theorem}[Variance stability of neural control variates]\label{thm:variance-stability}
Under Assumption~\ref{ass:nondegenerate-variance}, there exists a constant
$C>0$, depending only on
$M_H,M_L,\sigma_0,\norm{Q_H}_{L^2}$, and $\norm{Q_L}_{L^2}$, such that
\begin{equation}\label{eq:variance-stability-optimal}
\abs{\Var(\mathcal{E}_N^{*}[Q_H])-\Var(\mathcal{E}_L^{*}[Q_H])}
\le
C\bigl(
\norm{Q_N-Q_L}_{L^2}+\norm{Q_N-Q_L}_{L^2}^2
\bigr),
\end{equation}
where $\mathcal{E}_L^*$ and $\mathcal{E}_N^*$ denote the estimators
\eqref{eq:ideal-cv-estimator} with the optimal coefficients
$\lambda_L^*$ and $\lambda_N^*$, respectively.
\end{theorem}

\begin{proof}
By the optimal variance formula~\eqref{eq:optvar},
it is enough to estimate
\[
        \frac{\Cov(Q_H,Q_N)^2}{\Var(Q_N)}
        -
        \frac{\Cov(Q_H,Q_L)^2}{\Var(Q_L)}.
\]
Set
\[
        A_i:=\Cov(Q_H,Q_i),
        \qquad
        B_i:=\Var(Q_i),
        \qquad i\in\{L,N\}.
\]
Then
\[
\begin{aligned}
\abs{\frac{A_N^2}{B_N}-\frac{A_L^2}{B_L}}
&\le
\frac{\abs{A_N^2-A_L^2}}{B_N}
+
A_L^2
\abs{\frac1{B_N}-\frac1{B_L}}.
\end{aligned}
\]
By Assumption~\ref{ass:nondegenerate-variance}, $B_L$ and $B_N$ are bounded
away from zero. Moreover, Lemma~\ref{lem:cov-perturb} gives
\[
        \abs{A_N-A_L}
        \le
        \norm{Q_H-\E[Q_H]}_{L^2}\norm{Q_N-Q_L}_{L^2}.
\]
Hence
\[
\begin{aligned}
\abs{A_N^2-A_L^2}
=
\abs{A_N-A_L}\,\abs{A_N+A_L}
\le
C\bigl(
\norm{Q_N-Q_L}_{L^2}+\norm{Q_N-Q_L}_{L^2}^2
\bigr),
\end{aligned}
\]
where we used
\[
        \abs{A_N+A_L}
        \le
        2\abs{A_L}+\abs{A_N-A_L}.
\]
Finally,
\[
        \abs{\frac1{B_N}-\frac1{B_L}}
        =
        \frac{\abs{B_N-B_L}}{B_NB_L}
        \le
        C\bigl(
        \norm{Q_N-Q_L}_{L^2}+\norm{Q_N-Q_L}_{L^2}^2
        \bigr),
\]
by Lemma~\ref{lem:var-perturb}. Combining these estimates with
\eqref{eq:optvar} proves
\eqref{eq:variance-stability-optimal}.
\end{proof}
Thus, the neural surrogate is not required to approximate the high-fidelity observable directly. Instead, it must approximate the low-fidelity observable that acts as the ideal control. If this approximation is accurate in $L^2(\Omega)$, then the optimal variance reduction is stable under the replacement of $Q_L$ with $Q_N$. The estimate therefore quantifies the effect of the neural approximation on the optimal variance and shows that it is controlled by the observable error induced at the low-fidelity level.

\subsection{Multiple neural control variates}\label{sec:multiple-cv}
The scalar analysis above describes the stability of a single neural control variate. In many kinetic applications, however, variance reduction is obtained from several correlated low-fidelity observables used simultaneously. These observables may correspond to different reduced models, different moment components, or different asymptotic levels of the same multiscale hierarchy. In this case the optimal control-variate coefficient is no longer a scalar, but a vector determined by the covariance matrix of the controls. The purpose of this section is to show that the perturbative stability mechanism of Theorem~\ref{thm:variance-stability} persists in this multi-control setting: if the vector of neural control variates is close to the vector of
exact low-fidelity control variates, then the corresponding optimal
coefficients and optimal variance remain close to those obtained with
the exact low-fidelity controls.
Let
\[
        Q_H\in L^2(\Omega),\qquad
        \mathbf Q_i=(Q_i^{1},\ldots,Q_i^{m})^T\in L^2(\Omega;\R^m),
        \qquad i\in\{L,N\},
\]
where each component $Q_i^{\ell}$ is a scalar control observable, typically of the form
\[
        Q_i^{\ell}(z)=\mathcal Q\!\bigl(f_i^{\ell}(z)\bigr),
        \qquad \ell=1,\dots,m.
\]
Here $\mathbf Q_L$ denotes the vector of exact low-fidelity control variates and $\mathbf Q_N$ the corresponding vector in which one or more components have been replaced by neural surrogates. For a coefficient vector $\boldsymbol{\lambda}\in\R^m$, define
\begin{equation}\label{eq:multiple-cv-estimator}
\mathcal E_i^{\boldsymbol{\lambda}}[Q_H]
:=
\frac1{M_H}\sum_{k=1}^{M_H}\Bigl(Q_H^{(k)}-\boldsymbol{\lambda}^T \mathbf Q_i^{(k)}\Bigr)
+
\boldsymbol{\lambda}^T\frac1{M_L}\sum_{j=1}^{M_L}\widetilde{\mathbf Q}_i^{(j)},
\qquad i\in\{L,N\}.
\end{equation}
The two sample families are independent, and the samples in the first family are coupled through the same realization of the uncertain input. Now, the estimator \eqref{eq:multiple-cv-estimator} is unbiased for every $\boldsymbol{\lambda}$, i.e.
\[
        \E[\mathcal E_i^{\boldsymbol\lambda}[Q_H]]=\E[Q_H].
\]
The structure of  \eqref{eq:multiple-cv-estimator} is the direct analogue of the scalar estimator. The first Monte Carlo average uses coupled samples of the high-fidelity observable and of all control observables, so that their covariances can be exploited. The second Monte Carlo average estimates the expectations of the controls, using a larger and independent sample set. The coefficient vector $\boldsymbol{\lambda}$ determines the optimal linear combination of the controls. When $m=1$, this formulation reduces exactly to the two-level estimator of Section \ref{sec21}.
Let
\[
        \mathbf C_i:=\Cov(\mathbf Q_i,\mathbf Q_i)\in\R^{m\times m},
        \qquad
        \mathbf b_i:=\Cov(Q_H,\mathbf Q_i)\in\R^m.
\]
Then the variance of the control-variate estimator is
\begin{equation}\label{eq:multiple-cv-variance}
\Var(\mathcal E_i^{\boldsymbol\lambda}[Q_H])
=
\frac1{M_H}\Var(Q_H-\boldsymbol{\lambda}^T\mathbf Q_i)
+
\frac1{M_L}\boldsymbol{\lambda}^T \mathbf C_i\boldsymbol{\lambda}.
\end{equation}
If $\mathbf C_i$ is positive definite, the optimal coefficient is
\begin{equation}\label{eq:multiple-cv-lambda}
        \boldsymbol{\lambda}_i^*
        =
        \frac{M_L}{M_H+M_L}\,\mathbf C_i^{-1}\mathbf b_i,
\end{equation}
and the optimal variance is
\begin{equation}\label{eq:multiple-cv-optvar}
        \Var(\mathcal E_i^*[Q_H])
        =
        \frac{\Var(Q_H)}{M_H}
        -
        \frac{M_L}{M_H(M_H+M_L)}\, \mathbf b_i^T \mathbf C_i^{-1}\mathbf b_i .
\end{equation}
As for the scalar case, we introduce the following assumption:

\begin{assumption}[Vector nondegeneracy]
\label{ass:vector-nondegenerate}
There exists $\sigma_0>0$ such that, for every
$\boldsymbol{\lambda}\in\mathbb R^m$,
\[
\boldsymbol{\lambda}^T\mathbf C_L\boldsymbol{\lambda}
\ge \sigma_0|\boldsymbol{\lambda}|^2,
\qquad
\boldsymbol{\lambda}^T\mathbf C_N\boldsymbol{\lambda}
\ge \frac{\sigma_0}{2}|\boldsymbol{\lambda}|^2.
\]
\end{assumption}

\begin{lemma}[Difference of covariance matrices]\label{lem:matrix-cov-perturb}
There exists a constant $C>0$, depending only on the number of multifidelity controls $m$ and $\norm{\mathbf Q_L}_{L^2(\Omega;\R^m)}$, such that
\begin{equation}\label{eq:C-perturb}
        \norm{\mathbf C_N-\mathbf C_L}
        \le
        C(\norm{\mathbf Q_N-\mathbf Q_L}_{L^2(\Omega;\R^m)}+\norm{\mathbf Q_N-\mathbf Q_L}_{L^2(\Omega;\R^m)}^2).
\end{equation}
Moreover,
\begin{equation}\label{eq:b-perturb}
        \norm{\mathbf b_N-\mathbf b_L}
        \le
        \norm{Q_H-\E [Q_H]}_{L^2(\Omega)}\,\norm{\mathbf Q_N-\mathbf Q_L}_{L^2(\Omega;\R^m)} .
\end{equation}
Here $\norm{\cdot}$ denotes the Euclidean norm for vectors and the induced spectral norm for matrices.
\end{lemma}

\begin{proof}
For any $\mathbf a\in\R^m$ with $|\mathbf a|=1$,
\[\norm{\mathbf a^T(\mathbf Q_N-\mathbf Q_L)}_{L^2}\le \norm{\mathbf Q_N-\mathbf Q_L}_{L^2(\Omega;\R^m)},\qquad
\norm{\mathbf a^T\mathbf Q_L}_{L^2}\le \norm{\mathbf Q_L}_{L^2(\Omega;\R^m)},\] we obtain
\[
        \abs{\mathbf a^T(\mathbf C_N-\mathbf C_L)\mathbf a}
        \le C(\norm{\mathbf Q_N-\mathbf Q_L}_{L^2(\Omega;\R^m)}+\norm{\mathbf Q_N-\mathbf Q_L}_{L^2(\Omega;\R^m)}^2).
\]
The spectral-norm estimate \eqref{eq:C-perturb} follows by the variational characterization of symmetric matrices. For \eqref{eq:b-perturb}, note that
\[
        \mathbf b_N-\mathbf b_L=\Cov(Q_H,\mathbf Q_N-\mathbf Q_L).
\]
Thus, for every unit vector $\mathbf a$,
\begin{align*}
        &\abs{\mathbf a^T(\mathbf b_N-\mathbf b_L)}
        =\abs{\Cov(Q_H,\mathbf a^T(\mathbf Q_N-\mathbf Q_L))}
        \le
        \norm{Q_H-\E [Q_H]}_{L^2}\,\norm{\mathbf a^T(\mathbf Q_N-\mathbf Q_L)}_{L^2}\\
        &\le
        \norm{Q_H-\E [Q_H]}_{L^2}\,\norm{\mathbf Q_N-\mathbf Q_L}_{L^2(\Omega;\R^m)}.
        \end{align*}
Taking the supremum over $\mathbf a$ proves the result.
\end{proof}

\begin{theorem}[Multiple neural control variates]\label{thm:multiple-cv-stability}
Under Assumption~\ref{ass:vector-nondegenerate}, and for $\norm{\mathbf Q_N-\mathbf Q_L}_{L^2(\Omega;\R^m)} \le 1$, there exists a constant $C>0$ such that
\begin{equation}\label{eq:lambda-vector-perturb}
        \norm{\boldsymbol{\lambda}_N^*-\boldsymbol{\lambda}_L^*}
        \le
        C(\norm{\mathbf Q_N-\mathbf Q_L}_{L^2(\Omega;\R^m)}+\norm{\mathbf Q_N-\mathbf Q_L}_{L^2(\Omega;\R^m)}^2),
\end{equation}
and
\begin{equation}\label{eq:multiple-cv-var-perturb}
	\begin{aligned}
        &\abs{\Var(\mathcal E_N^*[Q_H])-\Var(\mathcal E_L^*[Q_H])}
        \le\\
        &C(\norm{\mathbf Q_N-\mathbf Q_L}_{L^2(\Omega;\R^m)}+\norm{\mathbf Q_N-\mathbf Q_L}_{L^2(\Omega;\R^m)}^2).
        \end{aligned}
\end{equation}
\end{theorem}

\begin{proof}
From \eqref{eq:multiple-cv-lambda},
\[
\boldsymbol{\lambda}_N^*-\boldsymbol{\lambda}_L^*
=
\frac{M_L}{M_H+M_L}
\bigl(\mathbf C_N^{-1}\mathbf b_N-\mathbf C_L^{-1}\mathbf b_L\bigr).
\]
We write
\[
\mathbf C_N^{-1}\mathbf b_N-\mathbf C_L^{-1}\mathbf b_L
=
\mathbf C_N^{-1}(\mathbf b_N-\mathbf b_L)+(\mathbf C_N^{-1}-\mathbf C_L^{-1})\mathbf b_L.
\]
By Assumption~\ref{ass:vector-nondegenerate},
$\norm{\mathbf C_N^{-1}}\le 2/\sigma_0$ and
$\norm{\mathbf C_L^{-1}}\le 1/\sigma_0$. Moreover,
\[
\mathbf C_N^{-1}-\mathbf C_L^{-1}
=
\mathbf C_N^{-1}(\mathbf C_L-\mathbf C_N)\mathbf C_L^{-1}.
\]
Using Lemma~\ref{lem:matrix-cov-perturb} gives \eqref{eq:lambda-vector-perturb}. Using the optimal representation \eqref{eq:multiple-cv-optvar}, we get for the variance
\[
\Var(\mathcal E_N^*[Q_H])-\Var(\mathcal E_L^*[Q_H])
=
-\frac{M_L}{M_H(M_H+M_L)}
\left(\mathbf b_N^T\mathbf C_N^{-1}\mathbf b_N-\mathbf b_L^T\mathbf C_L^{-1}\mathbf b_L\right).
\]
The difference in parentheses is bounded by
\[
\abs{(\mathbf b_N-\mathbf b_L)^T\mathbf C_N^{-1}\mathbf b_N}
+
\abs{\mathbf b_L^T(\mathbf C_N^{-1}-\mathbf C_L^{-1})\mathbf b_N}
+
\abs{\mathbf b_L^T\mathbf C_L^{-1}(\mathbf b_N-\mathbf b_L)}.
\]
The inverse bounds and Lemma~\ref{lem:matrix-cov-perturb} yield the result. Indeed, by \eqref{eq:b-perturb},
\[
\|\mathbf b_N-\mathbf b_L\|
\le
\|Q_H-\E[Q_H]\|_{L^2(\Omega)}
\|\mathbf Q_N-\mathbf Q_L\|_{L^2(\Omega;\R^m)}.
\]
Since $\|\mathbf Q_N-\mathbf Q_L\|_{L^2(\Omega;\R^m)}\le1$, both $\mathbf b_L$ and $\mathbf b_N$ remain bounded by a constant depending only on $\|Q_H\|_{L^2}$, $\|\mathbf Q_L\|_{L^2(\Omega;\R^m)}$, and $m$.
\end{proof}

\section{Kinetic surrogate estimates}\label{sec:applications}
As already stated, in our analysis and in the related applications, the neural network approximates an intermediate reduced kinetic model, such as a
Fokker--Planck or BGK surrogate, rather than the high-fidelity Boltzmann or Landau
dynamics \eqref{eq:kinetic} directly. This choice is not only motivated by the lower computational cost of the
reduced models. It also makes the neural approximation more feasible, since the
training residual does not involve the full Boltzmann or Landau collision operator. Indeed, for the models considered below, the  dynamics is either differential in the
velocity variable or local through a Maxwellian relaxation term, and can therefore be incorporated in a neural residual by automatic differentiation techniques. The neural surrogate is thus trained at the reduced
kinetic level, possibly calibrated or enriched with high-fidelity data~\cite{PareschiChenDimarco,ChenDimarcoPareschiSAPNN}.

\subsection{Problem setting and observables}\label{sec:observables}

In this section, we restrict ourselves to a perturbative regime near the normalized global Maxwellian
\[
M(v)=(2\pi)^{-d_v/2}e^{-|v|^2/2},
\qquad v\in\mathbb R^{d_v},
\]
in the same spirit as the hypocoercive framework with random input developed in
\cite{ZhuJin2017VPFP,JinZhu2018VPFP}.
Let
\[
H:=L^2\bigl(\mathbb T^{d_x}\times\mathbb R^{d_v};M^{-1}(v)\,dx\,dv\bigr),
\qquad
\|g\|_H^2
:=
\int_{\mathbb T^{d_x}\times\mathbb R^{d_v}}
\frac{|g(x,v)|^2}{M(v)}\,dx\,dv,
\]
and define the associated velocity-dissipation space
\[
V:=\Bigl\{g\in H:\ \nabla_v\!\left(\frac{g}{M}\right)\in L^2(M\,dx\,dv)\Bigr\}.
\]
The associated velocity-dissipation seminorm is
\[
\|g\|_V^2
:=
\int_{\mathbb T^{d_x}\times\mathbb R^{d_v}}
M(v)\left|\nabla_v\!\left(\frac{g}{M}\right)\right|^2\,dx\,dv.
\]
For an integer $m\ge0$, we set
\[
\mathcal H^m
:=
\Bigl\{g:\ \partial_x^\alpha g\in H,\ \ |\alpha|\le m\Bigr\},
\qquad
\|g\|_{\mathcal H^m}^2
:=
\sum_{|\alpha|\le m}\|\partial_x^\alpha g\|_H^2,
\]
and
\[
\mathcal V^m
:=
\Bigl\{g:\ \partial_x^\alpha g\in V,\ \ |\alpha|\le m\Bigr\},
\qquad
\|g\|_{\mathcal V^m}^2
:=
\sum_{|\alpha|\le m}\|\partial_x^\alpha g\|_V^2.
\]
Throughout this section we assume
\[
m>\frac{d_x}{2}+1,
\]
so that $H^m(\mathbb T^{d_x})$ is an algebra and embeds continuously into
$W^{1,\infty}(\mathbb T^{d_x})$.

For a kinetic function $f$ with positive density and
temperature, we denote
\[
\rho_f:=\int_{\mathbb R^{d_v}}f\,dv,
\qquad
\rho_f u_f:=\int_{\mathbb R^{d_v}}vf\,dv,
\qquad
d_v\rho_f T_f:=\int_{\mathbb R^{d_v}}|v-u_f|^2f\,dv.
\]
The weighted space $H$ controls the polynomial velocity moments needed in the
definitions above. In the perturbative regime near $M$, and provided
$\rho_f$ stays uniformly away from zero and $T_f$ stays uniformly positive,
the maps
\[
f\mapsto \rho_f,\qquad f\mapsto u_f,\qquad f\mapsto T_f
\]
are locally Lipschitz from $\mathcal H^m$ into $H^m(\mathbb T^{d_x})$.

We now connect the control error used in Section~\ref{sec:setting} with
the error between the exact kinetic density of the reduced model and its neural approximation.
A natural class of quantities of interest in kinetic UQ consists of final-time moments. 
Given $\psi\in L^2(M\,dx\,dv)$, for
$f\in C([0,T];H)$ define
\begin{equation}\label{eq:qoi-final-time-moment}
\mathcal Q(f)
:=
\int_{\mathbb T^{d_x}\times\mathbb R^{d_v}}
\psi(x,v)f(T,x,v)\,dx\,dv.
\end{equation}
Then
\begin{equation}\label{eq:qoi-final-time-lipschitz}
\begin{split}
|\mathcal Q(f)-\mathcal Q(g)|
&\le
\|\psi\|_{L^2(M\,dx\,dv)}\,\|f-g\|_{L^\infty(0,T;H)}\\
&\le
\|\psi\|_{L^2(M\,dx\,dv)}\,\|f-g\|_{L^\infty(0,T;\mathcal H^m)}.
\end{split}
\end{equation}
Consequently, if
\[
Q_L(z)=\mathcal Q(f_L(\cdot;z)),
\qquad
Q_N(z)=\mathcal Q(f_N(\cdot;z)),
\]
then
\begin{equation}\label{eq:qoi-observable-error}
\|Q_N-Q_L\|_{L^2(\Omega)}
\le
\|\psi\|_{L^2(M\,dx\,dv)}
\|f_N-f_L\|_{L^2(\Omega;L^\infty(0,T;\mathcal H^m))}.
\end{equation}
Combining \eqref{eq:qoi-observable-error} with
Theorem~\ref{thm:variance-stability},
under Assumption~\ref{ass:nondegenerate-variance}, gives
\begin{equation}\label{eq:kinetic-to-variance}
\begin{aligned}
&\bigl|
\Var(\mathcal E_N^*[Q_H])
-
\Var(\mathcal E_L^*[Q_H])
\bigr|
\le
C\Bigl(
\|\psi\|_{L^2(M\,dx\,dv)}
\|f_N-f_L\|_{L^2(\Omega;L^\infty(0,T;\mathcal H^m))}
\\
&\hspace{2.5cm}
+
\|\psi\|_{L^2(M\,dx\,dv)}^2
\|f_N-f_L\|_{L^2(\Omega;L^\infty(0,T;\mathcal H^m))}^2
\Bigr).
\end{aligned}
\end{equation}
Thus the remaining task is to estimate $f_N-f_L$ in
$L^2(\Omega;L^\infty(0,T;\mathcal H^m))$ for the reduced kinetic model
approximated by the neural surrogate. This is the model-dependent part of the
analysis, carried out below for Fokker--Planck and BGK surrogates.

\begin{assumption}[Perturbative admissibility]\label{ass:perturbative}
In the sequel we assume that the exact density $f_L$ and the neural surrogate
$f_N$ belong to an admissible perturbative class near $M$. 
More precisely,
\[
f_L-M,\ f_N-M \in C([0,T];\mathcal H^m),
\]
and
\begin{equation}\label{eq:vfp-smallness}
\|f_L-M\|_{L^\infty(0,T;\mathcal H^m)}
+
\|f_N-M\|_{L^\infty(0,T;\mathcal H^m)}
\le \delta
\end{equation}
for some $\delta>0$ sufficiently small.

We also assume that the corresponding macroscopic fields are uniformly
admissible: there exists $c_0>0$ such that
\[
\rho_{f_L},\rho_{f_N}\ge c_0,
\qquad
T_{f_L},T_{f_N}\ge c_0,
\]
and that the moment maps are locally Lipschitz in the following sense:
\begin{equation}\label{eq:moment-lipschitz}
\|\rho_f-\rho_g\|_{H^m_x}
+
\|u_f-u_g\|_{H^m_x}
+
\|T_f-T_g\|_{H^m_x}
\le
C_{\rm mom}\|f-g\|_{\mathcal H^m}
\end{equation}
for all admissible $f,g$ in the perturbative class, where
\[
\|U\|_{H_x^m}^2
=
\sum_{|\alpha|\le m}
\|\partial_x^\alpha U\|_{L_x^2}^2
=
\sum_{|\alpha|\le m}
\int_{\mathbb T^{d_x}}
|\partial_x^\alpha U(x)|^2\,dx.
\]
In particular,
\begin{equation}\label{eq:moment-smallness}
\|u_f\|_{H^m_x}+\|T_f-1\|_{H^m_x}
\le C_{\rm mom}\|f-M\|_{\mathcal H^m}.
\end{equation}
Finally, when strong residuals are used, we assume that the surrogate is
sufficiently regular for the residuals below to belong to the stated spaces.
\end{assumption}

\subsection{Fokker--Planck surrogate}\label{sec:vfp}
Fokker--Planck reduced models arise naturally in collisional plasma applications governed by Vlasov--Poisson--Landau equations. They provide the intermediate kinetic level used in the neural surrogate constructions of~\cite{ChenDimarcoPareschiSAPNN}.
Here we ignore the presence of external fields and consider the nonlinear FP surrogate
\begin{equation}\label{eq:fp-low-fidelity}
\partial_t f_L + v\cdot\nabla_x f_L
=
\nu(z)\,\mathcal C_{\mathrm{FP}}(f_L),
\qquad
f_L(0,x,v;z)=f_0(x,v;z),
\end{equation}
where
\begin{equation}\label{eq:vfp-nonlinear-operator}
\mathcal C_{\mathrm{FP}}(f)
:=
\nabla_v\cdot\Bigl((v-u_f)f + T_f\nabla_v f\Bigr),
\qquad
0<\nu_*\le \nu(z)\le \nu^*.
\end{equation}
This is the nonlinear Fokker--Planck operator associated with local Maxwellian equilibria.
For nonlinear Fokker--Planck and Vlasov--Poisson--Fokker--Planck models with uncertainty, hypocoercive, asymptotic-preserving,
and micro--macro formulations in perturbative regimes are discussed, for instance,
in~\cite{ZhuJin2017VPFP,JinZhu2018VPFP,NonlinearVFP2025}. See also~\cite{DesvillettesVillani2001,DolbeaultMouhotSchmeiser2015}
for classical hypocoercivity estimates for Fokker--Planck and linear kinetic equations.

\begin{proposition}[Surrogate estimate for nonlinear FP]\label{prop:fp-surrogate}
Let $f_L$ solve~\eqref{eq:fp-low-fidelity}, and let $f_N$ be a surrogate solution with strong residual
\begin{equation}\label{eq:fp-residual}
\mathcal R_{\mathrm{FP}}(f_N;z)
:=
\partial_t f_N + v\cdot\nabla_x f_N - \nu(z)\,\mathcal C_{\mathrm{FP}}(f_N).
\end{equation}
Assume that Assumption~\ref{ass:perturbative} holds. In addition, assume
that
\[
f_L-M,\ f_N-M\in L^2(0,T;\mathcal V^m),
\qquad
\|f_N-M\|_{L^2(0,T;\mathcal V^m)}\le K_V
\]
uniformly with respect to the random parameter $z$.
Then, for $\delta>0$ sufficiently small,
there exists a constant $C_{\mathrm{FP}}
=
C\bigl(T,\nu_*,\nu^*,m,C_{\rm mom},c_0, K_V\bigr)>0$
such that
\begin{equation}\label{eq:fp-pointwise-est}
\begin{split}
\|f_N-f_L\|_{L^\infty(0,T;\mathcal H^m)}
&+
\sqrt{\nu_*}\,\|f_N-f_L\|_{L^2(0,T;\mathcal V^m)}
\\
&\le
C_{\mathrm{FP}}
\Bigl(
\|f_N(0)-f_0(z)\|_{\mathcal H^m}
+
\|\mathcal R_{\mathrm{FP}}(f_N;z)\|_{L^2(0,T;\mathcal H^m)}
\Bigr).
\end{split}
\end{equation}
Consequently,
\begin{equation}\label{eq:fp-l2omega-est}
\|f_N-f_L\|_{L^2(\Omega;L^\infty(0,T;\mathcal H^m))}
\le
C_{\mathrm{FP}}\bigl(\eta_{0,N}^{\mathrm{FP}}+\delta_N^{\mathrm{FP}}\bigr),
\end{equation}
where
\[
\eta_{0,N}^{\mathrm{FP}}
:=
\|f_N(0)-f_0(z)\|_{L^2(\Omega;\mathcal H^m)},
\qquad
\delta_N^{\mathrm{FP}}
:=
\|\mathcal R_{\mathrm{FP}}(f_N;z)\|_{L^2(\Omega;L^2(0,T;\mathcal H^m))}.
\]
\end{proposition}

\begin{proof}
Set
\[
e:=f_N-f_L,
\qquad
R_N:=\mathcal R_{\mathrm{FP}}(f_N;z).
\]
Subtracting the equations for $f_N$ and $f_L$ gives
\begin{equation}\label{eq:vfp-error-equation}
\partial_t e + v\cdot\nabla_x e
=
\nu(z)\bigl(\mathcal C_{\mathrm{FP}}(f_N)-\mathcal C_{\mathrm{FP}}(f_L)\bigr)
+
R_N,
\qquad
e|_{t=0}=f_N(0)-f_0(z).
\end{equation}

We decompose the nonlinear collision difference as
\begin{equation}\label{eq:vfp-collision-difference}
\begin{aligned}
\mathcal C_{\mathrm{FP}}(f_N)-\mathcal C_{\mathrm{FP}}(f_L)
&=
\nabla_v\cdot\bigl(v e+\nabla_v e\bigr)
-\nabla_v\cdot(u_{f_L}e)
-\nabla_v\cdot\bigl((u_{f_N}-u_{f_L})f_N\bigr)
\\
&\quad
+\nabla_v\cdot\bigl((T_{f_L}-1)\nabla_v e\bigr)
+\nabla_v\cdot\bigl((T_{f_N}-T_{f_L})\nabla_v f_N\bigr).
\end{aligned}
\end{equation}
The first term is the linear Fokker--Planck operator around the global Maxwellian $M$.

Let $\alpha$ be a multi-index with $|\alpha|\le m$. Apply $\partial_x^\alpha$ to
\eqref{eq:vfp-error-equation}, take the $H$-inner product with $\partial_x^\alpha e$,
and sum over $|\alpha|\le m$. By periodicity in $x$, the transport term vanishes. For the
linear Fokker--Planck part we have the exact identity
\[
\sum_{|\alpha|\le m}
\bigl\langle
\partial_x^\alpha\nabla_v\cdot(v e+\nabla_v e),
\partial_x^\alpha e
\bigr\rangle_H
=
-\|e\|_{\mathcal V^m}^2.
\]
We shall repeatedly use the standard Gaussian weighted estimate
\[
\|vg\|_H+\|\nabla_v g\|_H
\le
C\bigl(\|g\|_H+\|g\|_V\bigr),
\]
together with the corresponding estimate for the spatial derivatives
$\partial_x^\alpha g$, $|\alpha|\le m$.

We now estimate the nonlinear perturbative terms. Since $H^m_x$ is an algebra and the
moment maps are locally Lipschitz by~\eqref{eq:moment-lipschitz}, we have
\[
\|u_{f_L}\|_{H^m_x}
+
\|T_{f_L}-1\|_{H^m_x}
\le C\delta,
\]
and
\[
\|u_{f_N}-u_{f_L}\|_{H^m_x}
+
\|T_{f_N}-T_{f_L}\|_{H^m_x}
\le C\|e\|_{\mathcal H^m}.
\]
Using product estimates in $H^m_x$, integration by parts in $v$, and the definition of
$\mathcal V^m$, we obtain
\[
\left|
\sum_{|\alpha|\le m}
\bigl\langle
\partial_x^\alpha\nabla_v\cdot(u_{f_L}e),
\partial_x^\alpha e
\bigr\rangle_H
\right|
\le
C\delta
\bigl(
\|e\|_{\mathcal H^m}^2+\|e\|_{\mathcal V^m}^2
\bigr),
\]
and similarly
\[
\left|
\sum_{|\alpha|\le m}
\bigl\langle
\partial_x^\alpha\nabla_v\cdot((T_{f_L}-1)\nabla_v e),
\partial_x^\alpha e
\bigr\rangle_H
\right|
\le
C\delta
\bigl(
\|e\|_{\mathcal H^m}^2+\|e\|_{\mathcal V^m}^2
\bigr).
\]

For the term involving $u_{f_N}-u_{f_L}$, write
\[
(u_{f_N}-u_{f_L})f_N
=
(u_{f_N}-u_{f_L})M
+
(u_{f_N}-u_{f_L})(f_N-M).
\]
Since $\nabla_v M=-vM$, the Maxwellian part satisfies
\[
\left|
\sum_{|\alpha|\le m}
\bigl\langle
\partial_x^\alpha\nabla_v\cdot((u_{f_N}-u_{f_L})M),
\partial_x^\alpha e
\bigr\rangle_H
\right|
\le
C\|e\|_{\mathcal H^m}^2.
\]
The remaining part is perturbative. Using~\eqref{eq:vfp-smallness} and Young's inequality,
\[
\left|
\sum_{|\alpha|\le m}
\bigl\langle
\partial_x^\alpha\nabla_v\cdot((u_{f_N}-u_{f_L})(f_N-M)),
\partial_x^\alpha e
\bigr\rangle_H
\right|
\le
C\delta
\bigl(
\|e\|_{\mathcal H^m}^2+\|e\|_{\mathcal V^m}^2
\bigr).
\]
Consequently,
\[
\left|
\sum_{|\alpha|\le m}
\bigl\langle
\partial_x^\alpha\nabla_v\cdot((u_{f_N}-u_{f_L})f_N),
\partial_x^\alpha e
\bigr\rangle_H
\right|
\le
C\|e\|_{\mathcal H^m}^2
+
C\delta
\bigl(
\|e\|_{\mathcal H^m}^2+\|e\|_{\mathcal V^m}^2
\bigr).
\]

For the temperature term, write
\[
(T_{f_N}-T_{f_L})\nabla_v f_N
=
(T_{f_N}-T_{f_L})\nabla_v M
+
(T_{f_N}-T_{f_L})\nabla_v(f_N-M).
\]
The first contribution is bounded by
\[
\left|
\sum_{|\alpha|\le m}
\bigl\langle
\partial_x^\alpha\nabla_v\cdot((T_{f_N}-T_{f_L})\nabla_v M),
\partial_x^\alpha e
\bigr\rangle_H
\right|
\le
C\|e\|_{\mathcal H^m}^2.
\]
For the second contribution, one obtains, for every $\kappa>0$,
\begin{equation*}
\begin{aligned}
&\left|
\sum_{|\alpha|\le m}
\bigl\langle
\partial_x^\alpha\nabla_v\cdot((T_{f_N}-T_{f_L})\nabla_v(f_N-M)),
\partial_x^\alpha e
\bigr\rangle_H
\right|
\le\\
&\kappa\|e\|_{\mathcal V^m}^2
+
C_\kappa
\left(\delta + \|f_N-M\|_{\mathcal V^m}\right)^2
\|e\|_{\mathcal H^m}^2.
\end{aligned}
\end{equation*}
Thus
\begin{equation*}
\begin{aligned}
&\left|
\sum_{|\alpha|\le m}
\bigl\langle
\partial_x^\alpha\nabla_v\cdot((T_{f_N}-T_{f_L})\nabla_v f_N),
\partial_x^\alpha e
\bigr\rangle_H
\right|
\le\\
&C\|e\|_{\mathcal H^m}^2
+
\kappa\|e\|_{\mathcal V^m}^2
+
C_\kappa
\left(\delta + \|f_N-M\|_{\mathcal V^m}\right)^2
\|e\|_{\mathcal H^m}^2.
\end{aligned}
\end{equation*}

Finally, the residual term is estimated by Cauchy--Schwarz and Young:
\[
\left|
\sum_{|\alpha|\le m}
\bigl\langle
\partial_x^\alpha R_N,\partial_x^\alpha e
\bigr\rangle_H
\right|
\le
\frac12\|R_N\|_{\mathcal H^m}^2
+
\frac12\|e\|_{\mathcal H^m}^2.
\]

Collecting the estimates and using $\nu(z)\in[\nu_*,\nu^*]$, we obtain
\begin{equation*}
\begin{aligned}
&\frac12\frac{d}{dt}\|e(t)\|_{\mathcal H^m}^2
+
\nu_*\|e(t)\|_{\mathcal V^m}^2
\le
\nu^*( C\delta\|e(t)\|_{\mathcal V^m}^2\\
&+
\kappa\|e(t)\|_{\mathcal V^m}^2
+
C\bigl(1+(\delta + \|f_N(t)-M\|_{\mathcal V^m})^2\bigr)
\|e(t)\|_{\mathcal H^m}^2)
+
\frac12\|R_N(t)\|_{\mathcal H^m}^2.
\end{aligned}
\end{equation*}
Choosing first $\kappa>0$ and then $\delta>0$ sufficiently small,
depending only on $\nu_*$ and $\nu^*$, the $\mathcal V^m$ terms on the
right-hand side are absorbed into the left-hand side.
Hence
\[
\frac12\frac{d}{dt}\|e(t)\|_{\mathcal H^m}^2
+
\frac{\nu_*}{2}\|e(t)\|_{\mathcal V^m}^2
\le
\nu^*C\bigl(1+( \delta + \|f_N(t)-M\|_{\mathcal V^m})^2\bigr)
\|e(t)\|_{\mathcal H^m}^2
+
\frac12\|R_N(t)\|_{\mathcal H^m}^2.
\]
By the uniform bound $\|f_N-M\|_{L^2(0,T;\mathcal V^m)}\le K_V$, Gronwall's lemma yields
\[
\|e\|_{L^\infty(0,T;\mathcal H^m)}^2
+
\nu_*\|e\|_{L^2(0,T;\mathcal V^m)}^2
\le
C_{\mathrm{FP}}^2
\left(
\|e(0)\|_{\mathcal H^m}^2
+
\|R_N\|_{L^2(0,T;\mathcal H^m)}^2
\right),
\]
which gives~\eqref{eq:fp-pointwise-est}. Taking the $L^2(\Omega)$ norm, since the constant is independent of $z$, gives \eqref{eq:fp-l2omega-est}.
\end{proof}

\subsection{Boltzmann--BGK surrogate}\label{sec:bgk}
BGK models provide a classical reduced kinetic description for Boltzmann-type
dynamics and are the low-fidelity models used in the rarefied-gas neural
control-variate setting of \cite{PareschiChenDimarco}.
The reduced model is naturally written as
\begin{equation}\label{eq:bgk-low-fidelity}
\partial_t f_L + v\cdot\nabla_x f_L
=
\nu(z)\bigl(\mathcal M[f_L]-f_L\bigr),
\qquad
f_L(0,x,v;z)=f_0(x,v;z),
\end{equation}
where
\[
\mathcal M[f](x,v)
=
\frac{\rho_f(x)}
     {(2\pi T_f(x))^{d_v/2}}
\exp\left(
-\frac{|v-u_f(x)|^2}{2T_f(x)}
\right)
\]
is the local Maxwellian associated with the moments of $f$.
Near the global Maxwellian, local stability and hypocoercive estimates for
relaxation and BGK-type models are classical; see, for example,~\cite{Herau2006,AchleitnerArnoldCarlen2018}.

\begin{proposition}[Surrogate estimate for nonlinear BGK]\label{prop:bgk-surrogate}
Let $f_L$ solve~\eqref{eq:bgk-low-fidelity}, and let the neural surrogate $f_N$ have residual
\begin{equation}\label{eq:bgk-residual}
\mathcal R_{\mathrm{BGK}}(f_N;z)
:=
\partial_t f_N + v\cdot\nabla_x f_N - \nu(z)\bigl(\mathcal M[f_N]-f_N\bigr).
\end{equation}
Assume that Assumption~\ref{ass:perturbative} holds. Then there exists a constant
\[
C_{\mathrm{BGK}}
=
C(T,\nu_*,\nu^*,m,C_{\rm mom},c_0,\delta)>0
\]
such that
\begin{equation}\label{eq:bgk-pointwise-est}
\|f_N-f_L\|_{L^\infty(0,T;\mathcal H^m)}
\le
C_{\mathrm{BGK}}
\Bigl(
\|f_N(0)-f_0(z)\|_{\mathcal H^m}
+
\|\mathcal R_{\mathrm{BGK}}(f_N;z)\|_{L^2(0,T;\mathcal H^m)}
\Bigr).
\end{equation}
Consequently,
\begin{equation}\label{eq:bgk-l2omega-est}
\|f_N-f_L\|_{L^2(\Omega;L^\infty(0,T;\mathcal H^m))}
\le
C_{\mathrm{BGK}}\bigl(\eta_{0,N}^{\mathrm{BGK}}+\delta_N^{\mathrm{BGK}}\bigr),
\end{equation}
where
\[
\eta_{0,N}^{\mathrm{BGK}}
:=
\|f_N(0)-f_0(z)\|_{L^2(\Omega;\mathcal H^m)},
\qquad
\delta_N^{\mathrm{BGK}}
:=
\|\mathcal R_{\mathrm{BGK}}(f_N;z)\|_{L^2(\Omega;L^2(0,T;\mathcal H^m))}.
\]
\end{proposition}

\begin{proof}
Set
\[
e:=f_N-f_L,
\qquad
R_N:=\mathcal R_{\mathrm{BGK}}(f_N;z).
\]
Subtracting the equations gives
\[
\partial_t e + v\cdot\nabla_x e
=
\nu(z)\bigl(\mathcal M[f_N]-\mathcal M[f_L]-e\bigr)
+
R_N,
\qquad
e|_{t=0}=f_N(0)-f_0(z).
\]
In the admissible perturbative regime, the Maxwellian map is $C^1$ with
locally Lipschitz derivative from the macroscopic variables into
$\mathcal H^m$. A Taylor expansion around the global Maxwellian $M$
therefore gives
\[
\mathcal M[f_N]-\mathcal M[f_L]
=
\Pi e+\mathcal N_{\mathrm{BGK}}(e;f_L,f_N),
\]
where $\Pi$ is the orthogonal projection in $H$ onto the space generated
by the collision invariants. More precisely,
\[
\Pi e
=
M\left(
\rho_e+m_e\cdot v
+\frac{\theta_e}{2}\bigl(|v|^2-d_v\bigr)
\right),
\]
with
\[
\rho_e
=
\int_{\mathbb R^{d_v}}e\,dv,
\qquad
m_e
=
\int_{\mathbb R^{d_v}}ve\,dv,
\qquad
\theta_e
=
\frac{1}{d_v}
\int_{\mathbb R^{d_v}}
\bigl(|v|^2-d_v\bigr)e\,dv.
\]
The nonlinear remainder satisfies
\[
\|\mathcal N_{\mathrm{BGK}}(e;f_L,f_N)\|_{\mathcal H^m}
\le
C\delta\|e\|_{\mathcal H^m},
\]
where the constant depends on the admissibility bounds for $\rho_f$ and
$T_f$, on $m$, and on the local Lipschitz constants of the moment maps.

Hence the error equation becomes
\[
\partial_t e + v\cdot\nabla_x e
=
\nu(z)(\Pi e-e)
+
\nu(z)\mathcal N_{\mathrm{BGK}}(e;f_L,f_N)
+
R_N.
\]
Apply $\partial_x^\alpha$ for $|\alpha|\le m$, take the $H$-inner product with
$\partial_x^\alpha e$, and sum over $|\alpha|\le m$. By periodicity in $x$, the transport
term vanishes, while the linearized BGK part gives
\[
\sum_{|\alpha|\le m}
\langle \partial_x^\alpha(\Pi e-e),\partial_x^\alpha e\rangle_H
=
-\|(I-\Pi)e\|_{\mathcal H^m}^2.
\]
Therefore
\[
\frac12\frac{d}{dt}\|e(t)\|_{\mathcal H^m}^2
+
\nu_*\|(I-\Pi)e(t)\|_{\mathcal H^m}^2
\le
C \nu^*\delta \|e(t)\|_{\mathcal H^m}^2
+
\frac12\|R_N(t)\|_{\mathcal H^m}^2
+
\frac12\|e(t)\|_{\mathcal H^m}^2.
\]
Thus
\[
\frac{d}{dt}\|e(t)\|_{\mathcal H^m}^2
\le
C\|e(t)\|_{\mathcal H^m}^2
+
\|R_N(t)\|_{\mathcal H^m}^2,
\]
where $C$ depends on $\nu^*,m,C_{\rm mom},c_0,\delta$.
Gronwall's lemma then yields
\[
\|e\|_{L^\infty(0,T;\mathcal H^m)}^2
\le
C_{\mathrm{BGK}}^2
\left(
\|e(0)\|_{\mathcal H^m}^2
+
\|R_N\|_{L^2(0,T;\mathcal H^m)}^2
\right),
\]
which gives~\eqref{eq:bgk-pointwise-est}. Taking the $L^2(\Omega)$ norm gives
\eqref{eq:bgk-l2omega-est}.
\end{proof}

\subsection{Two-level variance and AP stability for neural control variates}
\label{sec:two-level-ap}

We now combine the observable estimates with
the stability bounds for the reduced kinetic surrogates, keeping explicit the
kinetic scaling parameter. Let $\varepsilon>0$ denote a Knudsen-type
parameter associated with the fluid limit. Equivalently, in the notation of
Sections~\ref{sec:vfp}--\ref{sec:bgk}, the collision frequency may be read as
$\nu(z)/\varepsilon$, with $0<\nu_*\leq \nu(z)\leq \nu^*$. The limit
$\varepsilon\to0$ corresponds to the strongly collisional regime in which
the kinetic dynamics converges to its limiting Euler model~\cite{Villani2002survey}.

For a fixed scalar observable $\mathcal Q$, we write
\[
        Q_H^\varepsilon=\mathcal Q(f_H^\varepsilon),
        \qquad
        Q_L^\varepsilon=\mathcal Q(f_L^\varepsilon),
        \qquad
        Q_N^\varepsilon=\mathcal Q(f_N^\varepsilon),
\]
for the high-fidelity, reduced, and neural observables, respectively. We also
denote by $Q_E$ the corresponding observable of the limiting Euler model.
The results below are stated for each fixed $\varepsilon>0$. Their constants
may depend on $\varepsilon$, unless the underlying stability estimates are
uniform in the fluid scaling.

\begin{corollary}[Variance bound for Fokker--Planck and BGK surrogates]
\label{cor:concrete-variance}
Let $\mathcal Q$ be the final-time moment functional
\eqref{eq:qoi-final-time-moment}, and assume that
\[
\Var(Q_L^\varepsilon)\ge \sigma_0>0,
\qquad
\Var(Q_N^\varepsilon)\ge \frac{\sigma_0}{2}.
\]
If either Proposition~\ref{prop:fp-surrogate} or
Proposition~\ref{prop:bgk-surrogate} applies to the scaled reduced model, then
there exists a constant $C_\varepsilon>0$ such that
\begin{equation}\label{eq:concrete-variance-bound}
\Var(\mathcal E_N^*[Q_H^\varepsilon])
\le
\Var(\mathcal E_L^*[Q_H^\varepsilon])
+
C_\varepsilon
\Bigl(
(\delta_N^\varepsilon+\eta_{0,N}^\varepsilon)
+
(\delta_N^\varepsilon+\eta_{0,N}^\varepsilon)^2
\Bigr),
\end{equation}
where
\[
(\delta_N^\varepsilon,\eta_{0,N}^\varepsilon)
=
\begin{cases}
(\delta_N^{{\rm FP},\varepsilon},\eta_{0,N}^{{\rm FP},\varepsilon}),
& \text{for the Fokker--Planck surrogate},\\
(\delta_N^{{\rm BGK},\varepsilon},\eta_{0,N}^{{\rm BGK},\varepsilon}),
& \text{for the BGK surrogate}.
\end{cases}
\] The constant $C_\varepsilon$ depends on $M_H$, $M_L$, $\|Q_H^\varepsilon\|_{L^2}$,
$\|Q_L^\varepsilon\|_{L^2}$, $\sigma_0$,
$\|\psi\|_{L^2(M\,dx\,dv)}$, and on the corresponding scaled perturbative
stability constant.
\end{corollary}

\begin{proof}
By \eqref{eq:qoi-observable-error},
\[
\|Q_N^\varepsilon-Q_L^\varepsilon\|_{L^2(\Omega)}
\le
\|\psi\|_{L^2(M\,dx\,dv)}
\|f_N^\varepsilon-f_L^\varepsilon\|_
{L^2(\Omega;L^\infty(0,T;\mathcal H^m))}.
\]
If Proposition~\ref{prop:fp-surrogate} applies, then
\[
\|f_N^\varepsilon-f_L^\varepsilon\|_
{L^2(\Omega;L^\infty(0,T;\mathcal H^m))}
\le
C_{\mathrm{FP}}^\varepsilon
\bigl(\delta_N^{{\rm FP},\varepsilon}
+\eta_{0,N}^{{\rm FP},\varepsilon}\bigr),
\]
whereas Proposition~\ref{prop:bgk-surrogate} gives
\[
\|f_N^\varepsilon-f_L^\varepsilon\|_
{L^2(\Omega;L^\infty(0,T;\mathcal H^m))}
\le
C_{\mathrm{BGK}}^\varepsilon
\bigl(\delta_N^{{\rm BGK},\varepsilon}
+\eta_{0,N}^{{\rm BGK},\varepsilon}\bigr).
\]
The conclusion follows by inserting these bounds into
Theorem~\ref{thm:variance-stability}.
\end{proof}

The estimate above is perturbative with respect to the exact
low-fidelity observable. It therefore guarantees that the neural replacement does not
significantly degrade the ideal two-level variance reduction whenever the
neural observable remains close to the exact observable of the reduced model. In
the scaled regime, this perturbative statement has a direct AP interpretation.

Assume that, at the level of the chosen observable,
\[
        Q_H^\varepsilon\to Q_E,
        \qquad
        Q_L^\varepsilon\to Q_E
        \qquad \hbox{in }L^2(\Omega),
\]
with $\Var(Q_E)>0$. This is the AP consistency of the high-fidelity and
low-fidelity kinetic observables with the limiting Euler observable. If, in
addition, the neural replacement is consistent with the reduced
model, in the sense that
\[
        \|Q_N^\varepsilon-Q_L^\varepsilon\|_{L^2(\Omega)}\to0,
\]
then
\[
        Q_N^\varepsilon\to Q_E
        \qquad \hbox{in }L^2(\Omega).
\]
Consequently,
\[
        \rho_{H,N}^\varepsilon\to1,
        \qquad
        \frac{\Cov(Q_H^\varepsilon,Q_N^\varepsilon)}
        {\Var(Q_N^\varepsilon)}
        \to1.
\]
The optimal two-level coefficient therefore satisfies
\[
        \lambda_N^{*,\varepsilon}
        =
        \frac{\Cov(Q_H^\varepsilon,Q_N^\varepsilon)}
        {\Var(Q_N^\varepsilon)}
        \frac{M_L}{M_H+M_L}
        \to
        \frac{M_L}{M_H+M_L}.
\]
The factor $M_L/(M_H+M_L)$ is due to the fact that the expectation of the
control is itself estimated by Monte Carlo. If this expectation is available
exactly, or formally in the limit $M_L\to\infty$, the limiting coefficient is
$1$. Thus, in the two-level setting, the AP property is reflected in the
preservation of the limiting correlation structure and in the convergence of
the optimal control-variate coefficient to its fluid-limit value.

This limiting statement is independent of whether the residual stability constant in Corollary~\ref{cor:concrete-variance} is uniform in $\varepsilon$. A uniform version of the bound
\eqref{eq:concrete-variance-bound} transfers the AP limiting correlation of
the reduced control to the neural control. The estimate does not preclude the
more favorable case, observed in calibrated or enriched surrogates, in which
the neural control has a stronger correlation with $Q_H^\varepsilon$ than
the original low-fidelity observable.

\begin{remark}
\label{rem:architecture-residual-errors}
In the neural surrogate setting, the residual error in
Corollary~\ref{cor:concrete-variance} should be interpreted as an aggregate
architecture-dependent quantity. For micro--macro tensor architectures, it may
contain contributions from the PDE residual, moment-consistency errors between
the macro-network and the kinetic moments, velocity truncation or quadrature
errors, and, in AP settings, macroscopic consistency errors. These
contributions depend on the training loss, quadrature rule, and discretization,
and are therefore not analyzed separately here. They enter the present theory
through the observable error estimated in \eqref{eq:qoi-observable-error} and
through the residual bounds propagated by
Corollary~\ref{cor:concrete-variance}. In particular, uniform control of these
contributions with respect to $\varepsilon$ is what allows the neural
observable to retain the AP limiting correlation structure of the reduced
model. A complete tensor-neural estimate would require combining
the present stability mechanism with low-rank or tree tensor approximation
theory~\cite{Bachmayr2023Acta} and with training or generalization estimates
for separable neural architectures~\cite{ChoNamYangYunHongPark2023}.
\end{remark}

\subsection{Three-level AP neural control-variate hierarchy}
\label{sec:three-level-ap}

We now include the limiting Euler observable explicitly as an additional
control and consider the hierarchy
\[
\left(Q_H^\varepsilon,Q_N^\varepsilon,Q_E\right).
\]
The Euler level describes the limiting fluid behavior, whereas the
intermediate neural kinetic level provides a correction away from the
fluid regime.

Let
\[
V_{N,E}^\varepsilon
:=
\min_{\boldsymbol{\lambda}\in\mathbb R^2}
\operatorname{Var}
\left(
\mathcal E_{N,E}^{\boldsymbol{\lambda}}
[Q_H^\varepsilon]
\right)
\]
denote the optimal variance obtained with the two controls
\((Q_N^\varepsilon,Q_E)\), and let
\[
V_E^\varepsilon
:=
\min_{\lambda\in\mathbb R}
\operatorname{Var}
\left(
\mathcal E_E^\lambda[Q_H^\varepsilon]
\right)
\]
be the optimal variance obtained with the Euler control alone.

\begin{proposition}[AP limit of the hierarchical variance]
\label{prop:ap-hierarchical-variance}
Assume that
\[
Q_H^\varepsilon\longrightarrow Q_E
\qquad\mbox{in }L^2(\Omega),
\qquad
\operatorname{Var}(Q_E)>0.
\]
Then
\[
V_{N,E}^\varepsilon
\longrightarrow
\frac{\operatorname{Var}(Q_E)}{M_H+M_L},
\qquad
V_E^\varepsilon
\longrightarrow
\frac{\operatorname{Var}(Q_E)}{M_H+M_L},
\qquad \varepsilon\to0.
\]
In particular,
\[
V_E^\varepsilon-V_{N,E}^\varepsilon\longrightarrow0.
\]
Moreover, the optimal coefficient associated with the Euler control
alone satisfies
\[
\lambda_E^{*,\varepsilon}
=
\frac{M_L}{M_H+M_L}
\frac{\operatorname{Cov}(Q_H^\varepsilon,Q_E)}
     {\operatorname{Var}(Q_E)}
\longrightarrow
\frac{M_L}{M_H+M_L}.
\]
If the expectation of the controls is known exactly, or formally
\(M_L\to\infty\), the limiting variance is zero and
\(\lambda_E^{*,\varepsilon}\to1\).
\end{proposition}

\begin{proof}
The optimal variance obtained with the Euler control alone is
\[
V_E^\varepsilon
=
\frac{\operatorname{Var}(Q_H^\varepsilon)}{M_H}
-
\frac{M_L}{M_H(M_H+M_L)}
\frac{\operatorname{Cov}(Q_H^\varepsilon,Q_E)^2}
     {\operatorname{Var}(Q_E)}.
\]
Since \(Q_H^\varepsilon\to Q_E\) in \(L^2(\Omega)\),
\[
\operatorname{Var}(Q_H^\varepsilon)
\longrightarrow\operatorname{Var}(Q_E),
\qquad
\operatorname{Cov}(Q_H^\varepsilon,Q_E)
\longrightarrow\operatorname{Var}(Q_E),
\]
and therefore
\[
V_E^\varepsilon
\longrightarrow
\frac{\operatorname{Var}(Q_E)}{M_H+M_L}.
\]

Since the Euler observable is one of the controls in the hierarchy,
\[
V_{N,E}^\varepsilon\le V_E^\varepsilon.
\]
On the other hand, by the optimal-variance formula
\eqref{eq:multiple-cv-optvar},
\[
V_{N,E}^\varepsilon
=
\frac{\Var(Q_H^\varepsilon)}{M_H}
-
\frac{M_L}{M_H(M_H+M_L)}
(\mathbf b^\varepsilon)^T
(\mathbf C^\varepsilon)^{-1}
\mathbf b^\varepsilon,
\]
where
\[
\mathbf b^\varepsilon
=
\begin{pmatrix}
\Cov(Q_H^\varepsilon,Q_N^\varepsilon)\\[2mm]
\Cov(Q_H^\varepsilon,Q_E)
\end{pmatrix},
\qquad
\mathbf C^\varepsilon
=
\begin{pmatrix}
\Var(Q_N^\varepsilon)
&
\Cov(Q_N^\varepsilon,Q_E)
\\[2mm]
\Cov(Q_N^\varepsilon,Q_E)
&
\Var(Q_E)
\end{pmatrix}.
\]
Since the covariance matrix of
\((Q_H^\varepsilon,Q_N^\varepsilon,Q_E)\) is positive semidefinite,
the corresponding Schur-complement inequality gives
\[
(\mathbf b^\varepsilon)^T
(\mathbf C^\varepsilon)^{-1}
\mathbf b^\varepsilon
\leq
\Var(Q_H^\varepsilon).
\]
Hence
\[
V_{N,E}^\varepsilon
\geq
\frac{\Var(Q_H^\varepsilon)}{M_H+M_L}.
\]
Passing to the limit $\varepsilon \to 0$ gives the convergence of
\(V_{N,E}^\varepsilon\), and consequently
\[
V_E^\varepsilon-V_{N,E}^\varepsilon\longrightarrow0.
\]
The convergence of \(\lambda_E^{*,\varepsilon}\) follows directly from
the optimal two-level coefficient formula.
\end{proof}

Thus, in the fluid limit, the Euler observable alone attains the same
asymptotic optimal variance as the full three-level hierarchy. The
intermediate neural kinetic control may provide additional variance
reduction away from equilibrium, but this improvement vanishes as
\(\varepsilon\to0\). Notice that no asymptotic assumption on
\(Q_N^\varepsilon\) is required for this conclusion.

\section{Numerical tests}
\label{sec:numerical-diagnostics}
The numerical tests are designed to illustrate the two main results of
the analysis. The first test examines how the error introduced by
replacing the exact low-fidelity control with its neural approximation
affects the optimal variance. The second test considers the
asymptotic-preserving hierarchy and illustrates the role of the Euler
observable in the fluid regime, while the neural kinetic control
provides a correction away from equilibrium. The full computational
setting and the neural architectures are described
in~\cite{PareschiChenDimarco,ChenDimarcoPareschiSAPNN}.
Estimator errors are measured by the discrete mean absolute difference between 
the estimated and reference means of the quantity of interest over the relevant 
physical or velocity grid, with the reference mean computed by Gauss quadrature in the random space.
The covariance terms and the optimal control-variate coefficients are
estimated from the same paired high- and low-fidelity realizations used
to construct the corresponding estimators.

\subsection{Variance stability for Boltzmann--BGK}

We consider the two-bumps test in~\cite{ChenDimarcoPareschiSAPNN}. The
high-fidelity observable is computed from the Boltzmann model, while the
low-fidelity control is obtained from the BGK model and $Q_N$ denotes its
neural approximation. Here, the quantity of interest $Q$ is the distribution function $f$.
To test the stability mechanism of
Theorem~\ref{thm:variance-stability}, we measure the discrete observable error
\[
\varepsilon_Q=\|Q_N-Q_L\|_{L^2(\Omega)}
\]
for neural surrogates of different accuracies, together with the corresponding
variation of the optimal variance,
\[
\Delta V
=
\left|
\Var(\mathcal E_N^*[Q_H])
-
\Var(\mathcal E_L^*[Q_H])
\right|.
\]
The norm used here is the discrete $L^2(\Omega)$ norm evaluated over
the random samples. The sample sizes are $M_H=5$ and $M_L=2500$, and the
optimal coefficients are computed from \eqref{eq:lambda-star}.

\begin{figure}[tb]
    \begin{center}
        \mbox{
        {\includegraphics[width = 0.45 \textwidth, trim=0 0 0 0,clip]{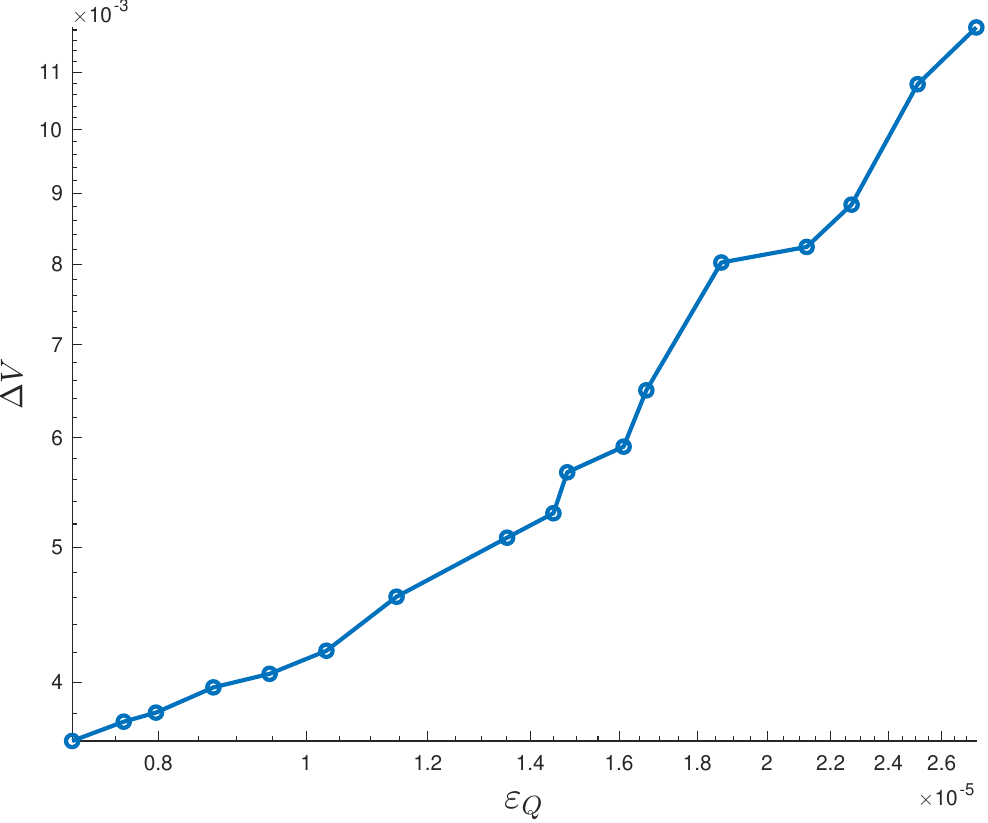}}
        {\includegraphics[width = 0.45 \textwidth, trim=0 0 0 0,clip]{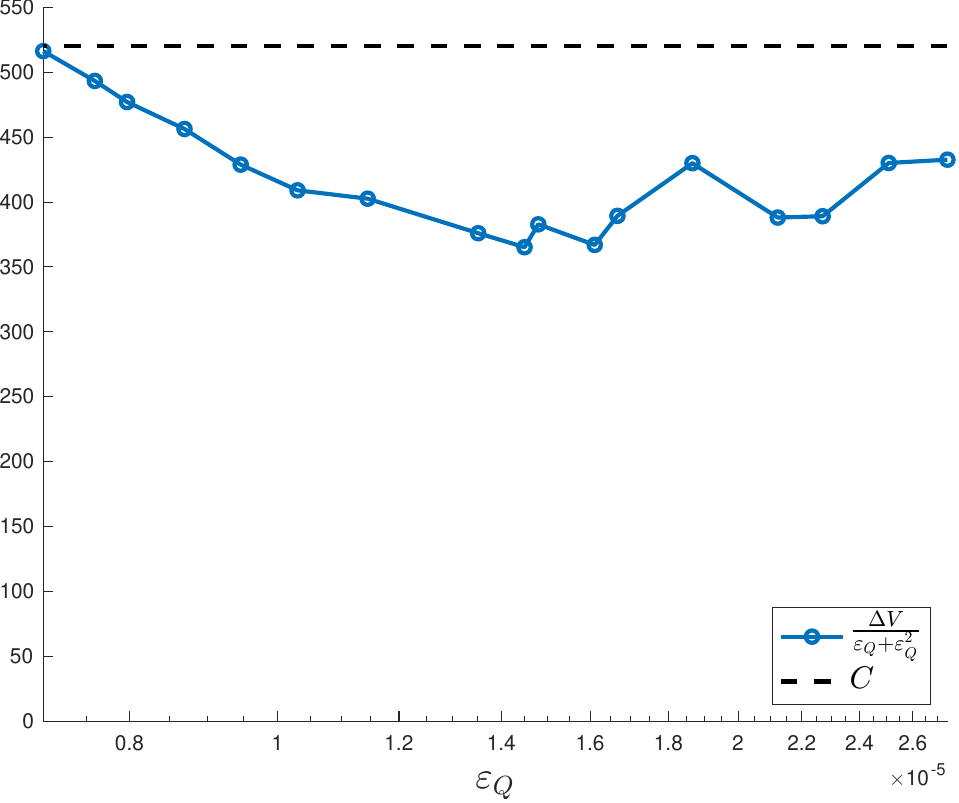}}}
        \caption{\sf Two-level variance-stability test for the Boltzmann--BGK neural control. Left: variation of the optimal variance $\Delta V(\varepsilon_Q)$ as a function of the observable error
$\varepsilon_Q$. Right: normalized ratio
$\Delta V/(\varepsilon_Q+\varepsilon_Q^2)$, compared with a constant bound.}
        \label{EEB}
    \end{center}
\end{figure}

The left panel of Fig.~\ref{EEB} shows $\Delta V$ for different neural surrogate errors $\varepsilon_Q$, indicating an approximately linear dependence over the range of errors considered. For the values reported, the right panel shows that ${\Delta V}/{(\varepsilon_Q + \varepsilon_Q^2)}$ remains below $C \approx 520$, consistently with the stability estimate of Theorem~\ref{thm:variance-stability}.

Figure~\ref{1Add} compares the performance of the Monte Carlo estimator with the exact and neural BGK control-variate estimators in the homogeneous Boltzmann--BGK test.
Both control-variate estimators reduce the error with respect to plain Monte Carlo, with the exact BGK control giving the largest improvement.
The neural BGK control exhibits the same variance-reduction trend, and the comparison between the two panels shows that a more accurate neural surrogate leads to a more effective estimator.

\begin{figure}[tb]
    \begin{center}
        \mbox{
        {\includegraphics[width = 0.45 \textwidth, trim=0 0 0 0,clip]{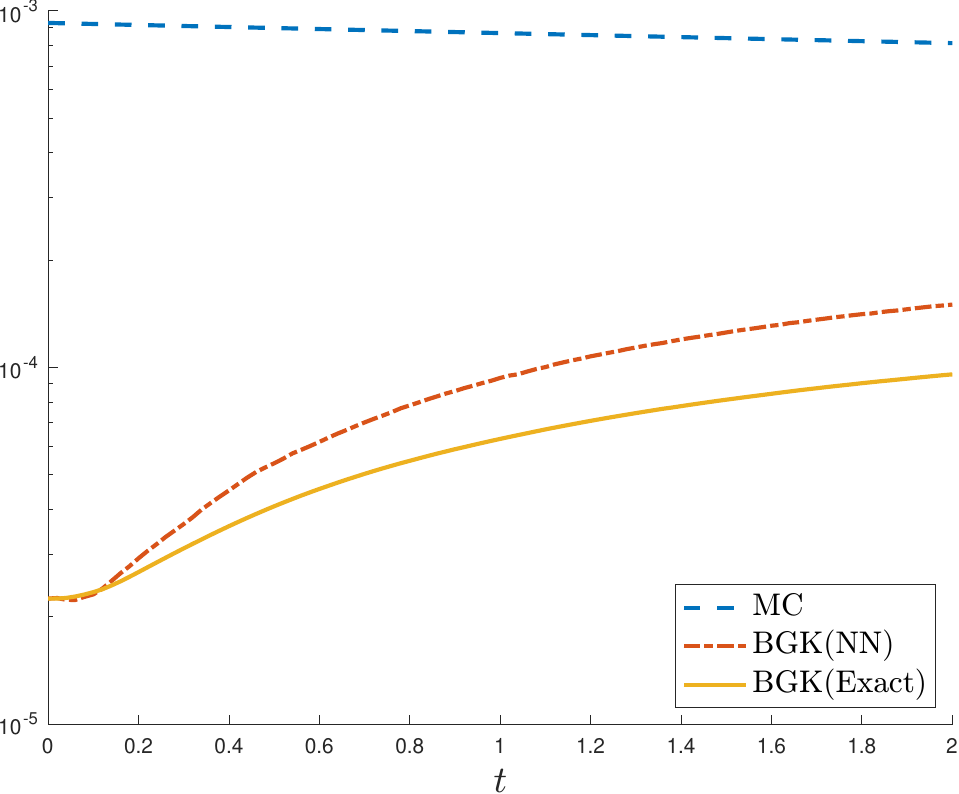}}
        {\includegraphics[width = 0.45 \textwidth, trim=0 0 0 0,clip]{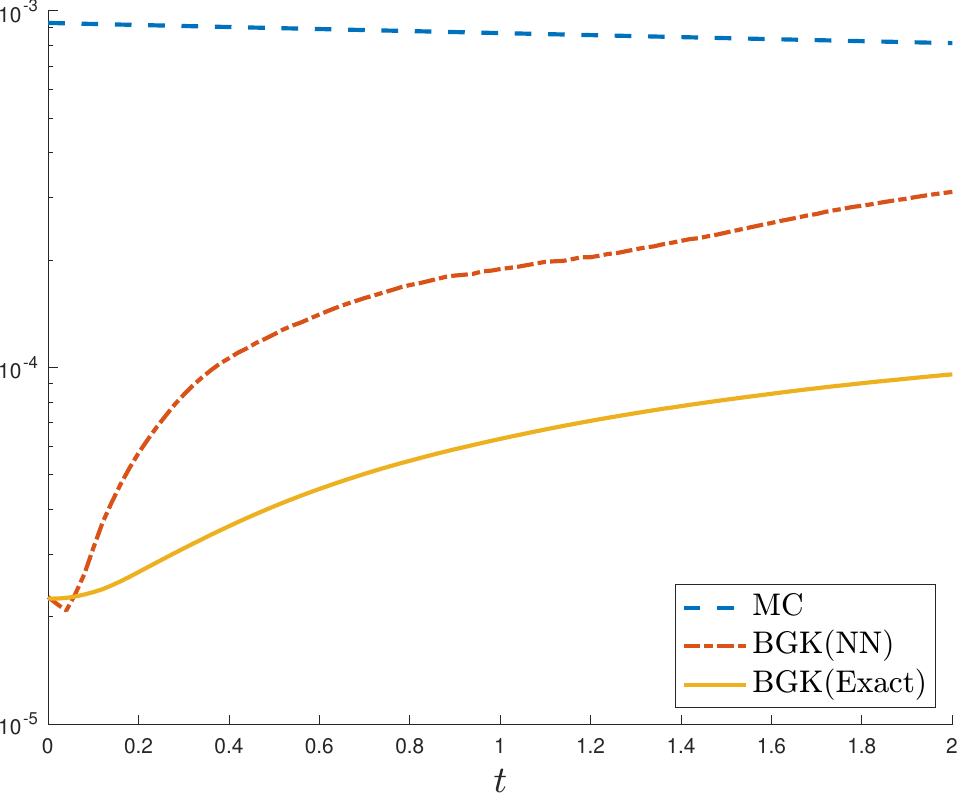}}}
        \caption{\sf Estimator performance for distribution function $f$ in the homogeneous Boltzmann--BGK test. Comparison of the Monte Carlo estimator, the neural BGK control variate, and the exact BGK control variate. Left: neural BGK surrogate with a final training loss of \(2.80\times10^{-6}\). Right: neural BGK surrogate with a final training loss of \(7.81\times10^{-6}\).}
        \label{1Add}
    \end{center}
\end{figure}

\subsection{A Vlasov--Fokker--Planck AP hierarchy test}
In the second experiment, following the asymptotic-preserving plasma setting of~\cite{ChenDimarcoPareschiSAPNN, PareschiChenDimarco}, we examine the three-level hierarchy
\[
        (Q_H^\varepsilon,\; Q_N^\varepsilon,\; Q_E)
\]
for decreasing values of the scaling parameter $\varepsilon$.
Here the estimator is applied to the macroscopic density $\rho$. The reported error norms are computed only afterward on the reconstructed expected density field. The high-fidelity observable $Q_H^\varepsilon$ is obtained from the Vlasov--Poisson--Landau (VPL) model, the neural control observable $Q_N^\varepsilon$ is obtained from a neural surrogate of the Vlasov--Poisson--Fokker--Planck (VPFP) model, and $Q_E$ is obtained from the Euler--Poisson (EP) model.

We consider the nonlinear Landau damping test in~\cite{PareschiChenDimarco} and compute the optimal coefficients $\lambda_E$ and $\lambda_D$ in the hierarchical basis
\[
        (Q_E,\; Q_N^\varepsilon-Q_E),
\]
with $M_H = 15$ and $M_L = 20000$.

The left panel of Fig.~\ref{Lambda} shows the optimal coefficients for
\[
\varepsilon\in\{10^{-4},10^{-3},10^{-2},10^{-1},1\}.
\]
Both the high-fidelity observable $Q_H^\varepsilon$ and the neural surrogate $Q_N^\varepsilon$ depend on $\varepsilon$. The results are consistent with the Euler-dominated behavior in the fluid limit, whereas the neural kinetic surrogate acts as a correction away from equilibrium.

In the right panel of Fig.~\ref{Lambda}, we fix $Q_N^\varepsilon \equiv Q_N^{1}$ and let the high-fidelity observable $Q_H^\varepsilon$ still depend on $\varepsilon$. The resulting optimal coefficients differ from those in the previous case, reflecting the lack of uniform alignment of the fixed kinetic control with the high-fidelity observable. The corresponding variance reduction is examined in Fig.~\ref{2Add}.

\begin{figure}[tb]
    \begin{center}
        \mbox{
        {\includegraphics[width = 0.45 \textwidth, trim=0 0 0 0,clip]{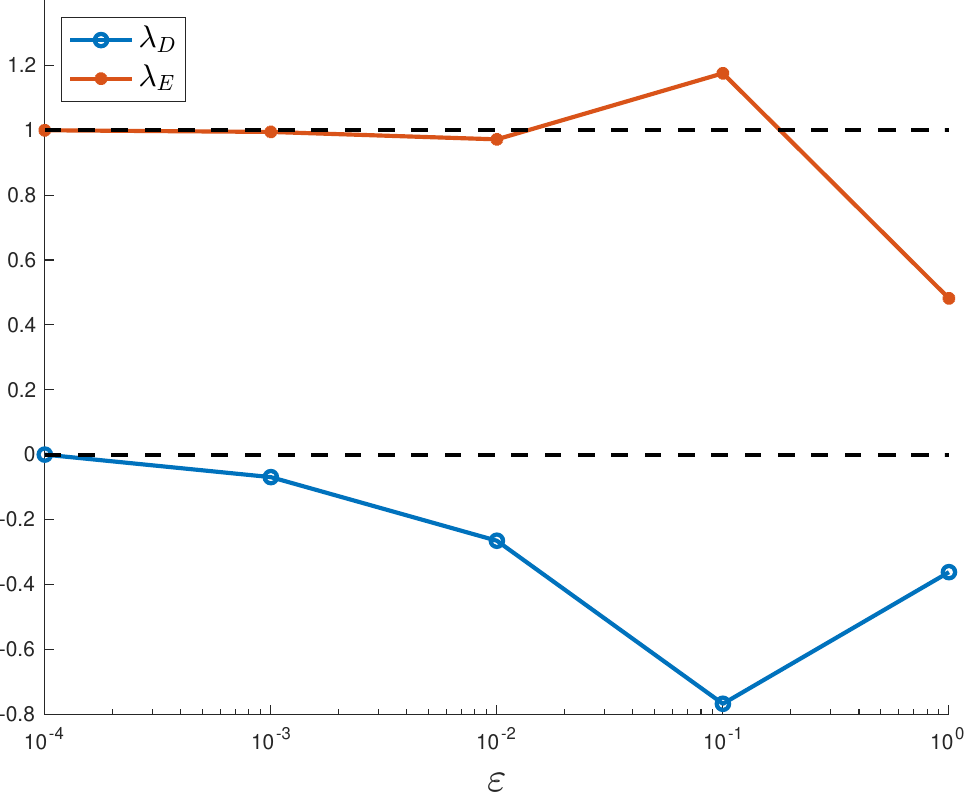}}
        {\includegraphics[width = 0.45 \textwidth, trim=0 0 0 0,clip]{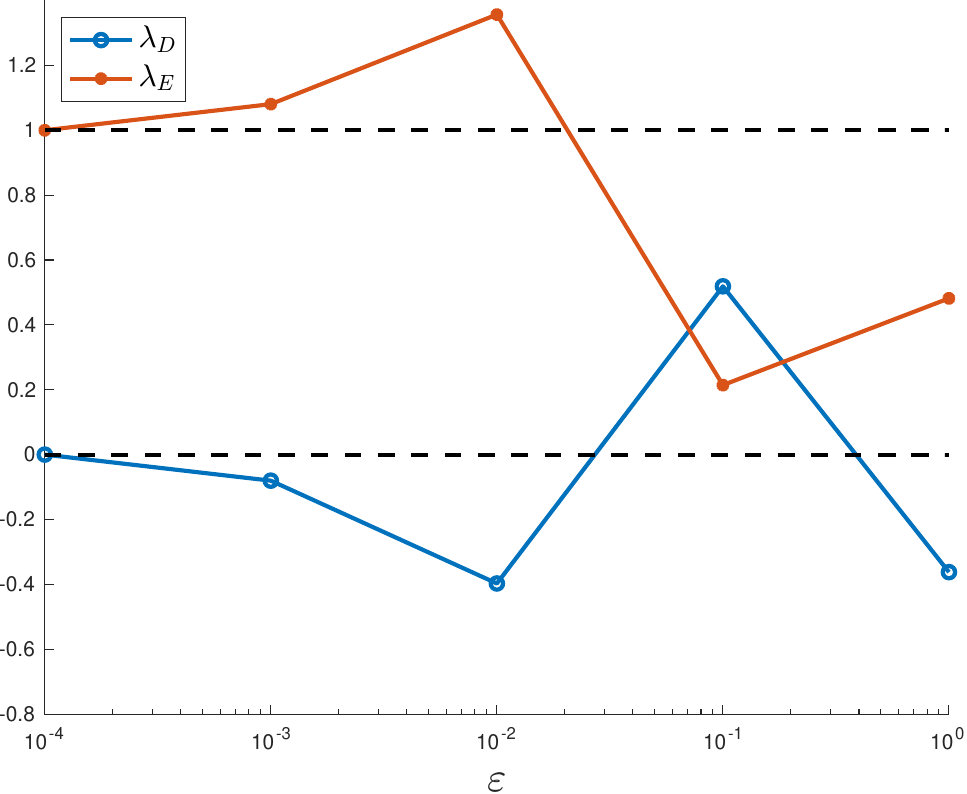}}}
        \caption{\sf Nonlinear Landau damping test for the VPFP neural control. Optimal hierarchical control variate coefficients $\lambda_D$ and $\lambda_E$ for different values of $\varepsilon$.
Left: $Q_N^\varepsilon$ is recomputed for each $\varepsilon$.
Right: $Q_N^\varepsilon\equiv Q_N^{1}$ is fixed.}
        \label{Lambda}
    \end{center}
\end{figure}

\begin{figure}[tb]
    \begin{center}
        \mbox{
        {\includegraphics[width = 0.45 \textwidth, trim=0 0 0 0,clip]{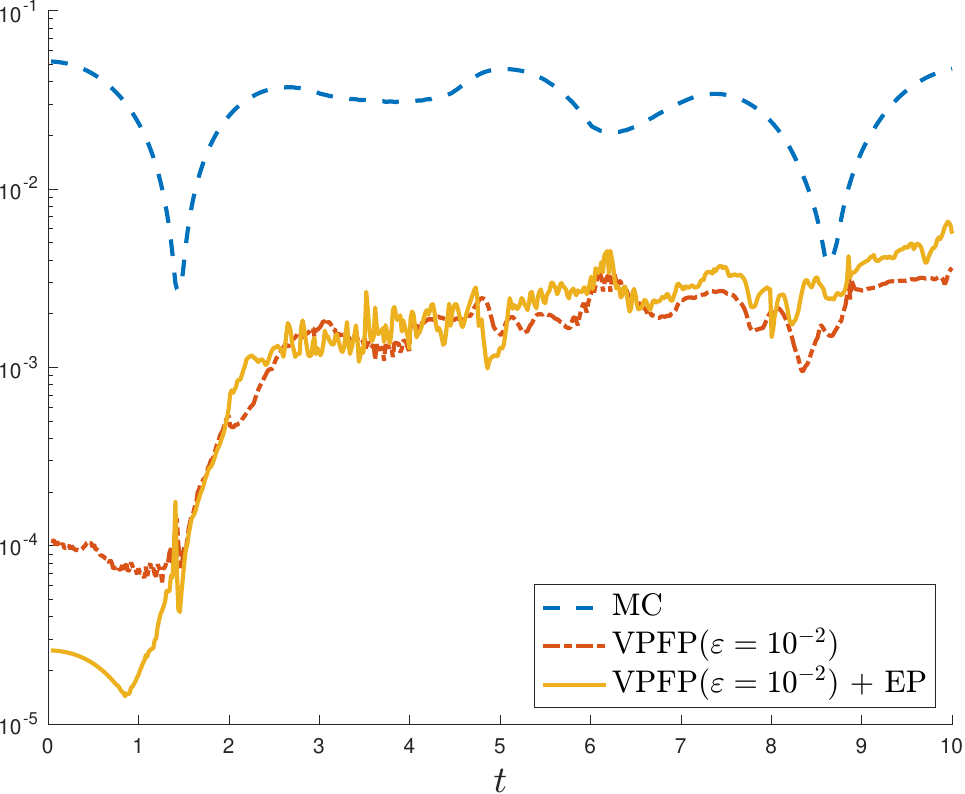}}
        {\includegraphics[width = 0.45 \textwidth, trim=0 0 0 0,clip]{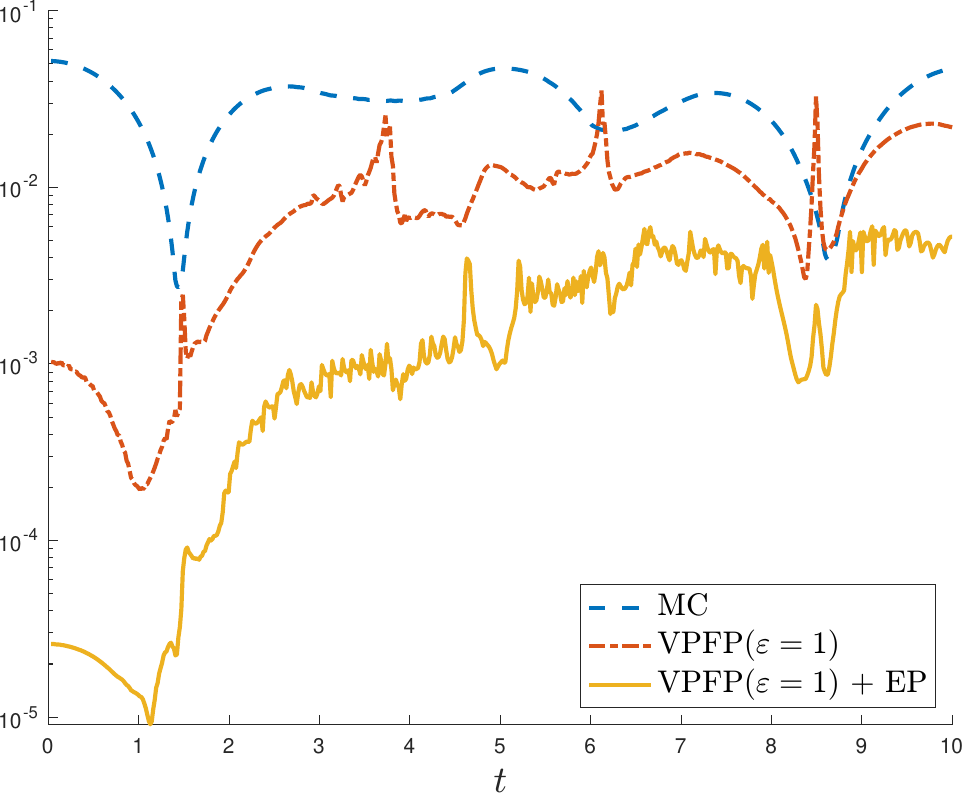}}}
        \caption{\sf Estimator errors for the macroscopic density $\rho$ in the VPL model with $\varepsilon=10^{-2}$.
Comparison of plain Monte Carlo, the neural VPFP control, and the
VPFP--EP hierarchical control. Left: the VPFP control uses the same
Knudsen number as the high-fidelity model. Right: the VPFP control is
computed with $\varepsilon=1$.}
        \label{2Add}
    \end{center}
\end{figure}

Figure~\ref{2Add} compares different uncertainty quantification estimators for the macroscopic density $\rho$ in the VPFP--EP hierarchy test, using the VPL solution at $\varepsilon=10^{-2}$ as the high-fidelity target.
The neural VPFP control variate reduces the error with respect to plain Monte Carlo when the same Knudsen number is used, as shown in the left panel.
When the control is computed instead at $\varepsilon=1$, the Knudsen-number mismatch weakens the efficiency of the kinetic control, but the estimator still improves over Monte Carlo.
The combined neural VPFP plus EP hierarchy control is particularly effective in this mismatched regime: the EP component captures the near-fluid behavior of the target solution and significantly enhances the variance reduction beyond that obtained with the kinetic control alone.
Since, for $\varepsilon=10^{-2}$, the VPFP dynamics are already close to the EP limit, the additional contribution of the EP hierarchy control is relatively limited.

Finally, we note that, as shown in our previous work~\cite{PareschiChenDimarco}, the model error may dominate the neural surrogate error at large times. This effect can be mitigated by incorporating additional high-fidelity data and using windowed training.

\section{Discussion and conclusions}
\label{sec:conclusion}
We have analyzed neural surrogate models used as approximate control observables in multifidelity uncertainty quantification for kinetic equations \cite{PareschiChenDimarco,ChenDimarcoPareschiSAPNN}. At the probabilistic level, replacing an exact low-fidelity observable $Q_L$ by an approximation $Q_N$ preserves unbiasedness and changes the optimal variance by an amount controlled by $\|Q_N-Q_L\|_{L^2(\Omega)}$. The estimate is perturbative with respect to the replacement of the
exact low-fidelity control variate by its neural approximation and does
not exclude the more favorable case in which a calibrated or enriched
neural surrogate is more strongly correlated with the high-fidelity
observable.

For final-time moment observables, the control error was connected with residual stability estimates for inhomogeneous nonlinear Fokker--Planck and BGK surrogates in a perturbative weighted Sobolev setting. The high-fidelity Boltzmann or Landau equation enters only through the observable and its covariance with the controls, whereas the stability estimate is required for the reduced model and its neural approximation. The extension to several controls yields a natural AP hierarchy. When the Euler observable is included among the controls, the optimal hierarchical variance converges to the variance associated with the Euler control in the fluid limit; the Euler coefficient converges to its limiting value and the contribution of the intermediate kinetic correction vanishes.

The numerical tests illustrate the variance stability under the replacement of $Q_L$ with $Q_N$ and the progressive concentration of the hierarchical control on the Euler level. For micro--macro tensor surrogates, the induced observable error and the residual bounds collect the errors associated with the residual, tensor approximation, moment consistency, and velocity discretization. A complete error theory would require combining the present analysis with approximation results for low-rank or tree tensor representations and with training or generalization estimates for separable neural architectures \cite{Bachmayr2023Acta,ChoNamYangYunHongPark2023}.

\subsection*{Acknowledgments}
The work of Wei Chen was partially supported by the China Scholarship Council, 
and the National Natural Science Foundation of China (NSFC) through the Research Fund for the Doctoral Program (No. 125B2021). 
Wei Chen also acknowledges the hospitality of the University of Ferrara. 
The work of Giacomo Dimarco was partially supported by the Italian Ministry of University and Research (MUR) through the PRIN
2020 project (No. 2020JLWP23) ``Integrated Mathematical Approaches to Socio–Epidemiological Dynamics”.
The work of Lorenzo Pareschi was supported by the Royal Society through the
Wolfson Fellowship ``Uncertainty quantification, data-driven simulations and
learning of multiscale complex systems governed by PDEs''. Lorenzo Pareschi also
acknowledges support from the FIS2023-01334 Advanced Grant ``Tackling
complexity: advanced numerical approaches for multiscale systems with
uncertainties'' (ADAMUS).

\end{document}